\documentclass[11pt]{article}
\usepackage[english]{babel}
\usepackage{amsthm}
\usepackage{flafter}
\usepackage{amsmath}
\usepackage{amssymb}
\usepackage{amscd,todonotes}
\usepackage{latexsym,microtype}
\usepackage{enumerate}
\usepackage{epsfig,euscript}
\usepackage{amsfonts,nicefrac}
\usepackage{xy,xspace,tikz,tikz-cd}
\usepackage{bbm,ifpdf}
\usepackage{setspace,mathrsfs}
\usepackage [margin=3.2cm]{geometry}

\ifpdf
  \usepackage{pdfsync}
	\usepackage[hidelinks]{hyperref}
	\usepackage{bookmark}
\fi
\newtheorem{theorem}{Theorem}[section]

\newtheorem{lemma}[theorem]{Lemma}
\newtheorem{proposition}[theorem]{Proposition}
\newtheorem{corollary}[theorem]{Corollary}

\newtheorem{conjecture}[theorem]{Conjecture}

{
\theoremstyle{definition}
\newtheorem{definition}[theorem]{Definition}
\newtheorem{example}[theorem]{Example}

\newtheorem{remark}[theorem]{Remark}
}

\newenvironment{proofample}{\noindent {\bf Proof of Proposition \ref{ample}:}}{\qed \par}

\newenvironment{proof again section}{\noindent {\bf Proof of Proposition \ref{again section}:}}{\qed \par}

\newcommand{\excise}[1]{}
\renewcommand{\and}{\qquad\text{and}\qquad}

\newcommand{\becircled}{\mathaccent "7017}

\newcommand{\Lie}{\operatorname{Lie}}

\newcommand{\C}{\mathbb{C}}
\newcommand{\Q}{\mathbb{Q}}
\newcommand{\Z}{\mathbb{Z}}

\newcommand{\R}{\mathbb{R}}

\newcommand{\cV}{\mathcal{V}}

\newcommand{\bA}{\mathbb{A}}

\newcommand{\cT}{\mathcal{T}}

\newcommand{\spec}{\operatorname{Spec}}
\newcommand{\bAE}{\bA^{\!E}}

\newcommand{\gm}{\mathbb{G}_m}
\newcommand{\Hom}{\operatorname{Hom}}
\newcommand{\Yreg}{Y^{\operatorname{reg}}}

\newcommand{\cA}{\mathcal{A}}
\newcommand{\cM}{\mathcal{M}}

\newcommand{\sign}{\operatorname{sign}}

\newcommand{\ba}{\mathbf{a}}

\newcommand{\cMp}{\cM_+}
\newcommand{\tv}{\hat{v}}
\newcommand{\Cext}{C_{\operatorname{ext}}}
\newcommand{\Cl}{\operatorname{Cl}}

\newcommand{\Pic}{\operatorname{Pic}}

\newcommand{\SF}{\operatorname{SF}}
\renewcommand{\div}{\operatorname{div}}

\begin{document}
\spacing{1.2}
\noindent{\Large\bf Hypertoric varieties and zonotopal tilings}\\

\noindent{\bf Matthew Arbo}\\
Department of Mathematics, University of Oregon,
Eugene, OR 97403\\

\noindent{\bf Nicholas Proudfoot}\footnote{Supported by NSF grant DMS-0950383.}\\
Department of Mathematics, University of Oregon,
Eugene, OR 97403\\

{\small
\begin{quote}
\noindent {\em Abstract.}
We give an abstract definition of a hypertoric variety, generalizing the existing constructive definition.
We construct a hypertoric variety associated with any zonotopal tiling, and we show that the previously known examples
are exactly those varieties associated with regular tilings.  
In particular, the examples that we construct from irregular tilings have not appeared before.
We conjecture that our construction gives a complete classification
of hypertoric varieties, analogous to the classification of toric varieties by fans.
\end{quote} }

\section{Introduction}
Hypertoric varieties were introduced by Bielawski and Dancer \cite{BD};
they are examples of conical symplectic varieties, which is currently an active topic of research at the
intersection of algebraic geometry \cite{Fu-survey,Kal09}, representation theory \cite{BMO,BLPWquant,BLPWgco}, and physics \cite{BDG-Coulomb}.
Hypertoric varieties are to conical symplectic varieties as toric varieties are to varieties:  they are the examples
with the largest possible abelian symmetry group, and they can be studied very explicitly via combinatorial methods.

The above analogy can be made into a formal definition of a hypertoric variety (as we will do below),
but this definition is not what appears in the literature, either in \cite{BD} or in numerous subsequent papers.  
Instead, what appears is a construction:
a recipe for building a variety out of a multi-arrangement of rational affine hyperplanes.  
This would be analogous
to defining a toric variety as a space that can be built out of a polyhedron via an explicit construction.
In the toric case, this approach would have a number of drawbacks.  We know that polyhedra classify
not toric varieties, but rather toric varieties equipped with an ample equivariant line bundle.  Projective toric varieties
have lots of different ample equivariant line bundles, while other toric varieties do not admit any.
Instead, one usually defines toric varieties abstractly and proves that they are classified by fans.
For any fan $\Sigma$, ample equivariant line bundles on the toric variety $X(\Sigma)$
are in bijection with polyhedra with normal fan $\Sigma$, of which there may be infinitely many or there may be none at all.

The purpose of this paper is to give an abstract definition of a hypertoric variety (Definition \ref{def}), and to
provide the hypertoric analogue of the classification of toric varieties by fans.  The answer, in brief, is that the hypertoric
analogue of a fan is a zonotopal tiling.  More precisely, given an integral zonotope that contains the origin in its interior
along with a tiling $\cT$ of that zonotope, we construct a variety $Y(\cT)$ and prove that it is hypertoric (Theorem \ref{main}).  We also prove that $\cT$ can be reconstructed from $Y(\cT)$ (Corollary \ref{reconstruct}), 
so that different tilings give rise to different hypertoric varieties, and we conjecture
that all hypertoric varieties arise via our construction (Conjecture \ref{all}).  Finally, we prove that, for any
tiling $\cT$ of an integral zonotope with the origin in its interior, ample equivariant line bundles on $Y(\cT)$
are in bijection with rational multi-arrangements of affine hyperplanes with normal tiling equal to $\cT$
(Corollary \ref{bundles}), of which there may be infinitely many or there may be none.\\

We now give some definitions so as to formulate our results and conjectures more precisely.
Fix an algebraically closed field $k$.  We call a variety $Y$ over $k$ {\bf convex} if
$k[Y]$ is a finitely generated integrally closed domain and the natural map $\pi\colon Y\to Y_0 := \spec k[Y]$ is 
proper and an isomorphism over the regular locus of $Y_0$.
When we talk about an ample line
bundle on $Y$, we will always mean ample relative to the map $\pi:Y\to Y_0$, so that $Y$ is projective over $Y_0$
if and only if it admits an ample line bundle.

If $Y$ is equipped with an action of the multiplicative group $\gm$, we say that $Y$ has {\bf non-negative weights}
if the induced grading of $k[Y]$ is non-negative, and we say that $Y$ has {\bf positive weights} if in addition the degree
zero piece consists only of constant functions.  Geometrically, this means that $Y_0$ is isomorphic
to the affine cone over a weighted projective variety.

Following Beauville \cite{Beau}, we call a normal Poisson variety $Y$ over $k$ {\bf symplectic} if
the Poisson structure on $Y$ is induced by a symplectic form $\omega\in\Omega^2(\Yreg)$
and for some (equivalently any) resolution $\nu:\tilde Y\to Y$, the form $\nu^*\omega$ extends to a 2-form on $\tilde Y$.
A symplectic $\gm$-variety $Y$ is called {\bf conical} if it is convex, it has positive weights, 
and the Poisson bracket has degree -2.

\begin{definition}\label{def}
Let $T$ be a torus.
A {\bf $T$-hypertoric variety} is a conical symplectic variety of dimension $2\dim T$,
equipped with an effective Hamiltonian action of $T$ that commutes with the action of $\gm$.
That is, it is a conical symplectic variety with the largest possible abelian symmetry group.
\end{definition}

Fix a torus $T$, and let $N := \Hom(\gm,T)$ be its cocharacter lattice.  
An {\bf integral zonotope} in $N_\R$ is the image of the cube $[-1,1]^E\subset \R^E$
along an affine linear map from $\R^E$ to $N_\R$ that takes $\Z^E$ to $N$.
Given an integral zonotope $Z\subset N_\R$,
we construct an affine symplectic $T\times\gm$-variety $Y(Z)$, where the action of $T$ is effective and Hamiltonian
and the Poisson bracket has degree -2 with respect to the action of $\gm$.  
The construction appears in Section \ref{chp:affinecase}, but the proof that $Y(Z)$ is symplectic does not come
until Proposition \ref{symp}.  We prove even later (Corollary \ref{pos wts}) that
the variety $Y(Z)$ has non-negative
weights if and only if $0\in Z$, and it has positive weights (and is therefore hypertoric) if and only if 0 lies in the interior
of $Z$.

Given a {\bf tiling} $\cT$ of $Z$ by smaller integral zonotopes
(see Section \ref{sec:combinatorics} for a precise definition and Figures 1 and 2 for examples), we construct
a $T$-hypertoric variety $Y(\cT)$ that is proper over $Y(Z)$ and via a map which is an isomorphism
over the regular locus (Corollary \ref{affinization}, Theorem \ref{main}).  If $\cT$ is the trivial tiling
consisting only of the zonotope $Z$ and its faces, then $Y(\cT) = Y(Z)$.  We conjecture that every $T$-hypertoric variety
arises via this construction (Conjecture \ref{all}).

A {\bf support function} on $\cT$ is a function $\varphi:Z\to\R$ that is linear on each element of $\cT$, takes lattice
points to the integers, and has the property that $\varphi(-\eta)=-\varphi(\eta)$ for all $\eta$ on the boundary of $Z$.
Such a support function is called {\bf strictly convex} if it is convex and the maximal tiles are the maximal domains of linearity.
If $\cT$ admits a strictly convex support function, it is called {\bf regular}.
We show that $Y(\cT)$ is projective over $Y(Z)$ if and only if $\cT$ is regular (Corollary \ref{regproj}),
and that the choices of a $T$-equivariant ample line bundle corresponds to the choice of a 
strictly convex support function on $\cT$ (Corollary \ref{bundles}).  See Figure 2 for an example of an irregular tiling.

As we explain in Section \ref{sec:combinatorics}, the set of pairs $(\cT,\varphi)$ consisting of a tiling along with a strictly convex
support function are in bijection with multi-arrangements of affine hyperplanes in $N_\R^*$ of the form
$\{m\in N_\R^*\mid m(a) + r = 0\}$, where $a\in N$ and $r\in\Z$.  Given such a multi-arrangement,
Bielawski and Dancer \cite{BD} and Hausel and Sturmfels \cite{HS} constructed a Poisson $T$-variety
equipped with an ample line bundle, and their construction coincides
with ours in this case (Remark \ref{GIT}).  Thus the construction of hypertoric varieties that appears in the literature
is a special case of our more general construction.  It is not obvious that the varieties
constructed in \cite{BD} and \cite{HS} are symplectic; this is clear in the smooth case, but nontrivial in the singular case.
We know of no proof of this fact other than as a special case of our Proposition \ref{symp}, in which we show
that $Y(\cT)$ is symplectic for any tiling $\cT$ (regular or not).\\

Most existing work on conical symplectic varieties includes the explicit hypothesis that the affinization map $\pi:Y\to Y_0$
is projective.  This never seemed to be much of a restriction, because there were no known examples for which $\pi$ was proper but
not projective.  This paper provides many such examples.  Indeed, it is in some sense the case that {\em most} zonotopal tilings
are irregular, and therefore that most hypertoric varieties are not projective over their affinizations.\\

We conclude the introduction with the following table, which summarizes the (conjectural) classification
of hypertoric varieties and its analogy with the classification of toric varieties.

\begin{table}[ht] \centering 
\begin{tabular}{|p{3.4cm}|p{3.6cm}||p{3.4cm}|p{4cm}|}\hline
Polyhedral structure & Toric object & Zonotopal structure & Hypertoric object \\ \hline \hline
pointed rational\;\;\;\;\;\;\;\; cone $\sigma\subset N_\R$ & affine $T$-toric variety $X(\sigma)$ & 
integral zonotope\;\;\;\;\;\; $Z\subset N_\R$ with $0\in\becircled{Z}$ & affine $T$-hypertoric\;\;\;\;\;\;\;\;\; variety $Y(Z)$ \\ \hline
rational fan $\Sigma$\;\;\;\;\;\;\;\;\; refining $\sigma$ & $T$-toric variety $X(\Sigma)$, proper over $X(\sigma)$ & tiling $\cT$ of $Z$ & 
$T$-hypertoric var. $Y(\cT)$, proper over $Y(Z)$ \\ \hline
(strictly convex) support function on $\Sigma$ & $T$-equivariant (ample) line bundle on $X(\Sigma)$ & 
(strictly convex) support function on $\cT$ & $T$-equivariant (ample)\;\;\; line bundle on $Y(\cT)$ \\ \hline
rational polyhedron in $N^*_\R$ with normal fan $\Sigma$ & $T$-equivariant ample line bundle on $X(\Sigma)$ &
multi-arrangement in $N^*_\R$ with normal tiling $\cT$
& $T$-equivariant ample line bundle on $Y(\cT)$ \\ \hline
\end{tabular}
\end{table}

\section{Combinatorial background}\label{sec:combinatorics}
The purpose of this section is to review the combinatorial definitions and constructions that will
be needed in the rest of the paper.
Let $N$ be a lattice, $E$ a finite set, and $\ba\in N^E$ a tuple of elements of $N$.
We define $$Z(\ba) := \sum_{e\in E}[-1,1]\cdot \ba_e\subset N_\R.$$
Any subset of $N_\R$ of this form, or any integral translation of such a subset, is called an 
{\bf integral zonotope}.  If $\ba$ is a basis for $N$, then we call any translation of $Z(\ba)$ a {\bf cube}.
More generally, if $\ba$ is a basis for $N_\R$ and each entry of $\ba$ is primitive, then we call any translation of $Z(\ba)$
a {\bf parallelotope}.

An element of $\{+,-,0\}^E$ is called a {\bf sign vector}.  
For any $r\in\Z^E$, we have the corresponding sign vector $\sign(r)\in \{+,-,0\}^E$,
and we put $$\cV^* := \{\sign(m(\ba))\mid m\in N^*\}\subset\{+,-,0\}^E.$$
For any $\ba\in N^E$ and $u\in \{+,-,0\}^E$,
we define the zonotope $$Z(\ba,u) := \sum_{u_e = 0} [-1,1]\cdot \ba_e + \sum_{u_e = +}\ba_e - \sum_{u_e = -}\ba_e\subset Z(\ba).$$
Then $Z(\ba,u)$ is a face of $Z(\ba)$ if and only if $u\in\cV^*(\ba)$ \cite[7.17]{Ziegler}, and all faces of $Z(\ba)$ are of this form.\footnote{We will
adopt the convention that the empty set is {\em not} a face of any zonotope.}

\begin{remark}\label{psa}
For most of our applications, it will be convenient to require $\ba$ to be a {\bf primitive spanning configuration},
by which we mean that every entry of $\ba$ is a nonzero primitive vector and $\ba$ spans $N_\R$.  We observe
that every integral zonotope may be written in the form $Z(\ba,u)$ for some primitive spanning configuration $\ba$
and sign vector $u$.  
\end{remark}

\begin{remark}\label{modify}
The choice of $\ba$ and $u$ in Remark \ref{psa} is not unique:  we can modify $(\ba,u)$ by a signed permutation of $E$
to obtain a new pair $(\ba',u')$ with $Z(\ba',u')=Z(\ba,u)$.  We could also construct such a pair $(\ba',u')$ by appending 
a collection of vectors that add to zero to $\ba$ and appending a collection of all $+$ coordinates to $u$.  Conversely, if $Z(\ba',u')=Z(\ba,u)$,
then $(\ba',u')$ is related to $(\ba,u)$ by a sequence of operations of these two forms.
\end{remark}

\begin{remark}\label{central arrangements}
A primitive arrangement $\ba$ also determines a central integral multi-arrangement $\cA(\ba)$ of hyperplanes
in $N_\R^*$, where the hyperplanes are the perpendicular spaces to the vectors.  Every central integral multi-arrangement
in $N_\R^*$ arises in this way, and we have $\cA(\ba) = \cA(\ba')$ if and only if $\ba$ and $\ba'$ differ by a signed permutation.
Thus there is a canonical bijection between integral zonotopes $Z\subset N_\R$ centered at the origin
and central integral multi-arrangements in $N_\R^*$.  Requiring $\ba$ to be spanning has the effect of imposing the
condition that $\dim Z = \dim N_\R$ and the condition that the corresponding hyperplane arrangement is essential.
\end{remark}

Let $Z\subset N_\R$ be an integral zonotope.  A {\bf tiling} of $Z$ is a set $\cT$ of integral zonotopes
satisfying the following conditions:
\begin{itemize}
\item $Z = \bigcup_{F\in \cT} F$
\item If $F'\in\cT$ and $F$ is a face of $F'$, then $F'\in\cT$
\item If $F,F'\in\cT$ and $F\cap F'$ is nonempty, then $F\cap F'$ is a face of both $F$ and $F'$.
\end{itemize}

\begin{example}\label{boring}
For any zonotope $Z$, the set of faces of $Z$ is a tiling of $Z$.  This is the only tiling of $Z$ if and only if $Z$ is a parallelotope.
\end{example}

\begin{example}\label{ptwo}
Let $N=\Z^2$, and let $\ba$ consist of the three vectors $(1,0)$, $(0,1)$, and $(1,1)$.
Then $Z(\ba)$ is a hexagon and admits three distinct tilings.  The first one is the one described in Example \ref{boring},
consisting only of $Z$ and its faces; the other two are shown below.

\begin{figure}[ht]
\begin{center}
\includegraphics[width=5cm]{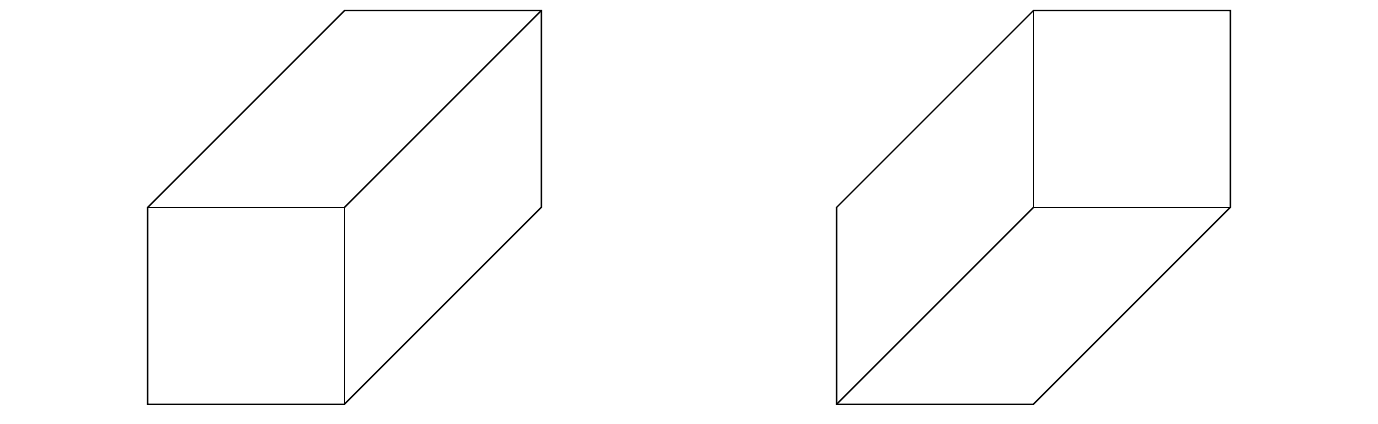}
\end{center}
\caption{Two tilings of a hexagon by cubes.}
\end{figure}

\noindent
More generally, suppose that $|E| = \operatorname{rk} N + 1$ and $\ba$ consists of $|E|$ vectors, any $|E|-1$ of which form a basis for $N$.
Then $Z(\ba)$ admits three distinct tilings, one of which is trivial and the other two of which have $|E|$ cubes as their
maximal tiles.
\end{example}

\begin{remark}
Every zonotope admits a tiling by parallelotopes; such tilings are precisely the ones that cannot be refined.
The zonotope $Z(\ba)$ admits a tiling by cubes if and only if $\ba$ is {\bf unimodular}, which means that for any subset $S\subset E$,
$\Z\{a_e\mid e\in S\} = N\cap \R\{a_e\mid e\in S\}$.  If $Z$ admits a tiling by cubes, then every tiling of $Z$ by parallelotopes is a tiling by cubes.
\end{remark}

We next describe three closely related constructions of tilings of the zonotope $Z(\ba)$, where $\ba$ is a primitive configuration
(we will not assume here that $\ba$ is spanning).
\\\\
{\bf Tuples of integers.}
Given an element $r\in\Z^E$, define a primitive configuration $\tilde\ba\in(N\oplus\Z)^E$ by putting $\tilde\ba_e = (\ba_e,r_e)$, and let 
$Z(\ba,r) := Z(\tilde\ba)\subset N_\R\oplus\R$ be its associated zonotope.  Then $Z(\ba)$ is the projection of
$Z(\ba,r)$ to $N_\R$.   Define a function $\psi(\ba,r):Z(\ba)\to\R$ given by the formula
$$\psi(\ba,r)(\eta) := \max\{s\mid(\eta,s)\in Z(\ba,r)\}.$$
Then $\psi(\ba,r)$ is a convex piecewise linear function on $Z(\ba)$, and the maximal domains of linearity
are the maximal elements of a tiling $\cT(\ba,r)$ of $Z(\ba)$.  
This is the tiling that you would ``see" if you looked down at $Z(\ba,r)$ from the point $(0,\infty)$.
Such a tiling of $Z(\ba)$ is called {\bf regular}.  The tilings in Example \ref{ptwo} are regular,
but the tiling of the larger hexagon in Figure \ref{irregular} is not \cite[Ex. 7.16]{Ziegler}.
\begin{figure}[ht]\label{irregular}
\begin{center}
\includegraphics[width=4cm]{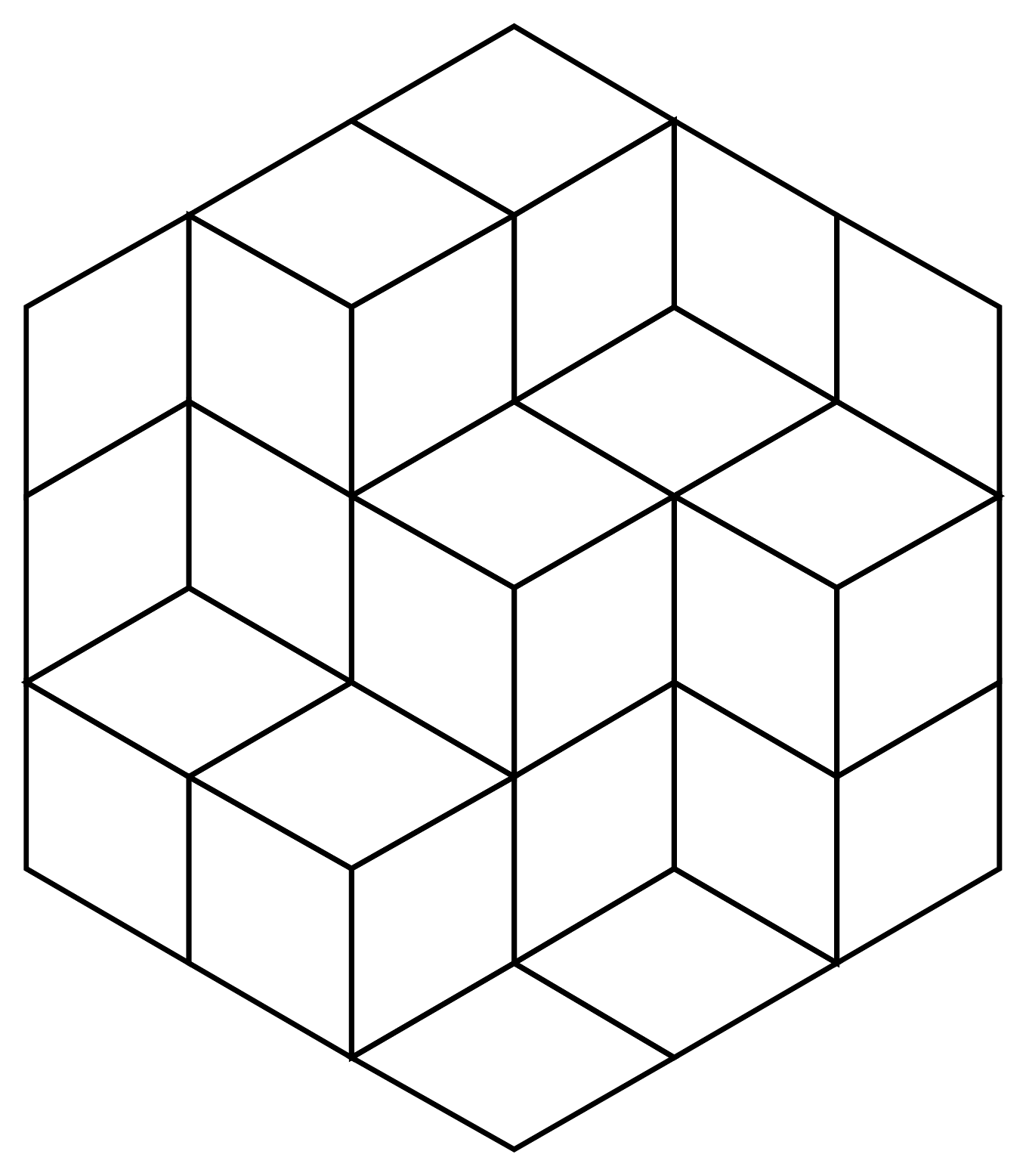}
\end{center}
\caption{An irregular tiling of a hexagon by cubes.}
\end{figure}\\\\
{\bf Affine arrangements.}
An element $r\in\Z^E$ also determines an integral affine multi-arrangement $\cA(\ba,r)$ in $N^*_\R$, obtained by intersecting
$\cA(\tilde\ba)$ with $N_\R^*\times\{1\}\subset N_\R^*\times\R$.  Equivalently, 
$\cA(\ba,r)$ is obtained by translating the hyperplanes
of $\cA(\ba)$ away from the origin, where the translations are determined by the coordinates of $r$.
Every integral affine multi-arrangement arises via this construction,
and $\cA(\ba,r) = \cA(\ba',r')$ if and only if the pairs $(\ba,r)$ and $(\ba',r')$ differ by a signed permutation of $E$.
Thus we obtain a surjective map from integral affine multi-arrangements in $N^*_\R$ to regular tilings in $N_\R$.
Furthermore, we have an inclusion-reversing bijection between the faces of $\cA(\ba,r)$ and the elements of $\cT(\ba,r)$.
This should be regarded as an analogue of the construction that takes a polyhedron in $N_\R^*$ to its normal
fan, which is a regular fan in $N_\R$, where we have an inclusion-reversing bijection between faces of the polyhedron
and cones in the fan.  Consequently, we refer to the zonotopal tiling associated with an integral affine multi-arrangement
as the {\bf normal tiling} to that arrangement. 
\\\\
{\bf Affine oriented matroids.}
Let $\tilde E := E\sqcup\{\infty\}$.  For any $u\in\subset\{+,-,0\}^E$ and $\epsilon\in\{+,-,0\}$, let 
$(u,\epsilon)$ be the corresponding element of $\{+,-,0\}^{\tilde E}$.
Let $\cM$ be an affine oriented matroid over $\cV^*(\ba)$;
that is, $\cM\subset\{+,-,0\}^{\tilde E}$ is a collection of sign vectors that satisfy the covector axioms of an oriented matroid
\cite[7.21]{Ziegler}, and $u\in\cV^*(\ba)$ if and only if $(u,0)\in\cM$.  Let $\cMp := \{u\mid (u,+)\in\cM\}$; 
this is called the set of
{\bf positive covectors} of $\cM$.  Then $\cT(\ba,\cM) := \{Z(\ba,u)\mid u\in\cMp\}$ is a tiling of $Z(\ba)$.  Furthermore,
the Bohne-Dress theorem \cite[7.32]{Ziegler} says that every tiling of $Z(\ba)$ arises in this manner.

\begin{remark}
Note that, for any $r\in\Z^E$, $\cM(\ba,r) := \cV^*(\tilde\ba)$ is an affine oriented matroid over $\cV^*(\ba)$, and
$\cT(\ba,\cM(\ba,r)) = \cT(\ba,r)$.
\end{remark}

\begin{remark}
Given $\ba\in N^E$ and $u\in\{+,-,0\}^E$, let 
$E_u := \{e\in E\mid u_e=0\}$ and let $\bar\ba$ be the restriction of $\ba$ to $E_u$.
For any $v\in \{+,-,0\}^{E_u}$, define $\hat v\in\{+,-,0\}^E$ by appending the nonzero entries of $u$ to $v$.  Then any affine oriented
matroid $\cM$ over $\cV^*(\bar\ba)$ determines a tiling $\cT(\ba,u,\cM) := \{Z(\ba,\hat v)\mid v\in\cMp\}$ of $Z(\ba,u)$, and the Bohne-Dress
theorem easily generalizes to say that every tiling of $Z(\ba,u)$ arises in this manner.
\end{remark}

Let $\ba\in N^E$ be a primitive spanning arrangement and let $\cT = \cT(\ba,\cM)$ be a tiling of $Z = Z(\ba)$.
A {\bf support function} on $\cT$ is a function $\varphi:Z\to\R$ that is affine linear on every element of $\cT$
and has the property that $\varphi(-\eta) = -\varphi(\eta)$ for all $\eta$ on the boundary of $Z$.
Let $\SF(\cT)$ be the group of support functions under addition; we will describe this group more explicitly.
Let $\Lambda\subset\Z^E$ be the kernel of the map $\Z^E\to N$ taking the coordinate vector indexed by $e$ to $\ba_e$.
For each positive covector $u\in\cMp$,
let $\Lambda_u := \Lambda\cap \Z^{E_u}\subset\Z^E$.
Let $$\Lambda_\cT \;:= \sum_{u\in\cMp}\Lambda_u.$$  Note that $\Lambda_\cT = 0$ if
and only if every element of $\cT$ is a parallelotope, and $\Lambda_\cT = \Lambda$ if and only if
$\cT$ consists only of $Z$ and its faces.
Let 
\begin{equation}\label{pct}
P_\cT := \Lambda_\cT^\perp = \{r\in\Z^E \mid \text{$r\cdot\lambda = 0$ for all $\lambda\in\Lambda$}\}
\end{equation} 
be the perpendicular space to $\Lambda$ with respect to the dot product.  For each $r\in P_\cT$, define
a function $\varphi(\ba,r):Z\to\R$ as follows:  for each $u\in\cMp$, any element $\eta\in Z(\ba,u)$ may be expressed
in the form $\eta = \sum t_e\ba_e$, where 
$$t_e\in
\begin{cases}
[-1,1] \;\;\;\text{if $u_e = 0$}\\
\{1\} \;\;\;\;\;\;\;\,\text{if $u_e = +$}\\
\{-1\} \;\;\;\;\;\text{if $u_e = -$}.
\end{cases}$$
For such an $\eta$, we define $\varphi(\ba,r)(\eta) := r\cdot t$.  The element $t\in \R^E$ is not uniquely determined
by $\eta$, but $P_\cT$ is defined precisely to ensure that $\varphi(\ba,r)(\eta)$ is well-defined.  In this way,
we obtain a group isomorphism $r\mapsto\varphi(\ba,r)$ from $P_\cT$ to the group $\SF(\cT)$ of support functions on $\cT$.

A support function $\varphi$ on $\cT$ is said to be {\bf strictly convex} if it is convex and the maximal
elements of $\cT$ are the maximal domains of linearity of $\varphi$.  We have already seen an example
of such a function:  
the function $\psi(\ba,r)$ is a strictly convex support function on the tiling $\cT(\ba,r)$.
In fact, all strictly convex functions have this form.  More precisely, $\varphi(\ba,r)$ is strictly convex on $\cT$
if and only if it is convex and $\cT = \cT(\ba,r)$, in which case $\varphi(\ba,r) = \psi(\ba,r)$.  
In particular, a tiling $\cT$ admits a strictly convex support function if and only if it is regular.

\begin{remark}\label{local fan}
For any pair $F\subset F'\in\cT$, we define the cone $\sigma_{F,F'}:= \R_{\geq 0}(F'-F)$, and for any $F\in \cT$,
we define the {\bf local fan} $\Sigma_F := \{\sigma_{F,F'}\mid F'\supset F\}$.  Geometrically, $\Sigma_F$ is the fan that you ``see" when you zoom in near a point in the relative interior of $F$.  If $\varphi$ is a strictly convex support
function on $\cT$, then it induces a strictly convex support function on the local fan $\Sigma_F$ for every $F\in\cT$.
In particular, if there exists an element $F\in\cT$ such that the fan $\Sigma_F$ does not admit a strictly convex
support function, then neither does $\cT$.
This is an obstruction to regularity of $\cT$, but it is not the only obstruction.
In the tiling shown in Figure \ref{irregular}, every $\Sigma_F$ admits a strictly convex support function
(this is always the case in dimension 2), but $\cT$ does not.
\end{remark}

\section{Affine varieties from zonotopes}\label{chp:affinecase}
Let $N$ be a lattice and let $T := \spec k[N^*]$ be the algebraic torus over $k$ with cocharacter lattice $\Hom(\gm,T)\cong N$.
The purpose of this section will be to assign to any integral zonotope $Z\subset N_\R$ a Poisson $T\times\gm$-variety $Y(Z)$.
We will proceed by defining a variety $Y(\ba, u)$ for any triple consisting of a 
finite set $E$, a primitive spanning configuration $\ba\in N^E$, and a sign vector $u\in\{+,-,0\}^E$; we will then show that 
$Y(\ba, u)$ depends only on the zonotope $Z(\ba, u)$.

Given a finite set $E$, consider the symplectic vector space $S(E) := \spec k[z_e, w_e]_{e\in E}$, 
with symplectic form $\sum_{e\in E}z_e\wedge w_e$.
The group $\gm^E\times\gm$ acts as follows:  if $t\in \gm^E$ and $s\in\gm$, then
$$(t,s)\cdot z_e = t_e^{-1} s z_e\and (t,s)\cdot w_e = t_es w_e.$$
The symplectic form has weight 2 with respect to the action of $\gm$, while the action of $\gm^E$
is Hamiltonian with moment map
$$\mu_E = (z_ew_e)_{e\in E}:S(E)\to\bAE\cong\Lie(\gm^E)^*.$$

Let $\ba\in N^E$ be a primitive spanning configuration, and
let $K$ be the kernel of the homomorphism $\gm^E\to T$ induced by $\ba$.
Let $\mu_\ba$ be the composition of $\mu_E$ with the projection $\Lie(\gm^E)^*\to\Lie(K)^*$, which is a moment
map for the action of $K$ on $S(E)$.  We define $Y(\ba)$ to be the affine symplectic quotient of $S(E)$
by $K$; that is, $$Y(\ba) := \spec k[\mu_\ba^{-1}(0)]^K.$$

\begin{remark}\label{affineBD}
Our definition of $Y(\ba)$ agrees with the original definition by Bielawski and Dancer \cite[\S 5]{BD}.
They work exclusively over $\C$, but their construction makes sense for arbitrary $k$.
\end{remark}

The group $\gm^E\times\gm$ acts on $Y(\ba)$ with kernel $K$, so this action descends to an effective action
of $T\times\gm$.  The Poisson bracket on $k[S(E)]$ induces a Poisson bracket on $k[\mu_\ba^{-1}(0)]^K$,
which has weight 0 with respect to the action of $T$ and weight -2 with respect to the action of $\gm$.
The action of $T$ is Hamiltonian with moment map $$\mu:Y(\ba)\to \ker\left(\Lie(\gm^E)^*\to\Lie(K)^*\right)\cong\Lie(T)^*$$
induced by $\mu_E$.  This particular choice of moment map is uniquely characterized by the condition
that it is homogenous (of weight 2) for the action of $\gm$.

\begin{proposition}\label{ya}
If $\ba$ and $\ba'$ are primitive spanning configurations in $N$ with $Z(\ba) = Z(\ba')$, then
$Y(\ba)\cong Y(\ba')$ as Poisson $T\times\gm$-varieties.
\end{proposition}

\begin{proof}
If $Z(\ba) = Z(\ba')$, then $\ba$ and $\ba'$ differ by a signed permutation of $E$.
This signed permutation induces a $\gm$-equivariant Poisson automorphism of $S(E)$
in which the coordinates are permuted and, if there is a sign in the $e$ coordinate, we send
$z_e$ to $w_e$ and $w_e$ to $-z_e$.  This automorphism is not $\gm^E$-equivariant;
rather, it takes the action of $K$ to the action of $K'$, and descends to a $T\times\gm$-equivariant
Poisson isomorphism from $Y(\ba)$ to $Y(\ba')$.
\end{proof}

Now let $u\in\{+,-,0\}^E$ be any sign vector.  Let
$$f_u := \prod_{u_e = +}z_e \cdot \prod_{u_e = -} w_e \in k[S(E)],$$
and let $S(E)_u := \spec k[S(E)]_{f_u}$ be the complement of the vanishing locus of $f_u$.
We define $Y(\ba, u)$ to be the affine quotient of $S(E)_u\cap\mu_{\ba}^{-1}(0)$ by $K$.
The inclusion of $S(E)_u$ into $S(E)$ defines a $T\times\gm$-equivariant morphism $Y(\ba,u)\to Y(\ba)$ of 
Poisson $T\times\gm$-varieties.

\begin{proposition}\label{include}
If $u\in\cV^*(\ba)$, then
the morphism $Y(\ba,u)\to Y(\ba)$ is an open inclusion of Poisson varieties.
\end{proposition}

\begin{proof}
Given a point $(p,q)\in S(E)$, let $p_e = z_e(p,q)$ and $q_e = w_e(p,q)$.
The morphism $Y(\ba,u)\to Y(\ba)$ fails to be an open inclusion if and only if there exists
a point $(p,q)\in S(E)_u\cap \mu_{\ba}^{-1}(0)$ and an element
$$\lambda \in \Lambda =\ker(\Z^E\to N) \cong \Hom(K,\gm)$$
such that $\lim_{t\to 0}\lambda(t)\cdot (p,q)$ exists and lies outside of $S(E)_u$.
Explicitly, this is equivalent to the following set of conditions:
\begin{itemize}
\item $\sum \lambda_e \ba_e = 0$
\item $\lambda_e <0 \;\;\Rightarrow\;\; p_e = 0 \;\;\Rightarrow\;\; u_e \neq +$
\item $\lambda_e >0 \;\;\Rightarrow\;\; q_e = 0 \;\;\Rightarrow\;\; u_e \neq -$
\item there exists an $e\in E$ such that $u_e\neq 0$ and $\lambda_e\neq 0$.
\end{itemize}
The hypothesis that $u\in\cV^*(\ba)$ means that there exists an element $m\in N^*$ such that $u_e$
is equal to the sign of $m(\ba_e)$ for all $e\in E$.  The first condition above tells us that
$\sum\lambda_e m(\ba_e) = 0$, and the last condition tells us that the individual terms of this sum are not all zero.
That means that there must be at least one positive term and at least one negative term.
However, the second and third conditions tell us that there can be no negative terms.
\end{proof}

\begin{remark}\label{lawrence}
We defined $Y(\ba, u)$ by first passing to the zero level of $\mu_\ba:S(E)_u\to\Lie(K)^*$ and then taking an affine
GIT quotient by $K$.  We could equally well have done this in the opposite order.  The affine
GIT quotient $X(\ba, u) := \spec k[S(E)_u]^K$ is called the {\bf affine Lawrence toric variety} associated
with the pair $(\ba, u)$.  The map $\mu_\ba$ descends to a map on $X(\ba, u)$,
and $Y(\ba, u)$ is isomorphic to the subvariety of $X(\ba)$ defined by the vanishing of this map.

Let us be a little bit more explicit about the combinatorics of the affine Lawrence toric variety, as we will
need to come back to it later in the paper.  
Consider the antidiagonal embedding of $\Lambda$ into $\Z^E\oplus\Z^E$, and let $\tilde N$
denote the quotient of $\Z^E\oplus\Z^E$ by $\Lambda$.  
For each $e\in E$, let $\rho_e^{\pm}\in \tilde N$
be the images of the two basis vectors of $\Z^E\oplus\Z^E$ indexed by $e$.
Then $X(\ba)$ is the $(\gm^E\times\gm^E)/K$-toric variety associated with the {\bf Lawrence cone} $\sigma(\ba,u)\subset\tilde N_\R$
spanned by the vectors $$\{-\rho_e^\epsilon\mid \epsilon\neq u_e\} 
\;\;=\;\; \{-\rho_e^{\pm}\mid u_e = 0\}\cup\{-\rho_e^+\mid u_e=-\}\cup\{-\rho_e^-\mid u_e=+\}.$$
\end{remark}

Given a $T\times\gm$-variety $Y$ and a cocharacter $\eta\in N$ of $T$, let $Y[\eta]$
be the $T\times\gm$-variety whose underlying variety is $Y$ on which the element
$(t,s)\in T\times\gm$ acts in the same way that $(\eta(s)^{-1}t,s)$ acts on $Y$.  Thus $Y[\eta]$ is
$T$-equivariantly isomorphic to $Y$, but this isomorphism cannot in general be made $T\times\gm$-equivariant.
We refer to $Y[\eta]$ as the {\bf twist} of $Y$ by $\eta$.

\begin{lemma}\label{append}
Let $E' = E \sqcup \{0\}$, let $\ba'\in N^{E'}$ be a primitive spanning configuration obtained by appending
$\eta$ to $\ba$, and let $u'\in\{+,-,0\}^{E'}$ be the sign vector obtained by appending $+$ to $u$.
Then $Y(\ba', u')$ is isomorphic to $Y(\ba,u)[\eta]$ as a Poisson $T\times\gm$-variety.
\end{lemma}

\begin{proof}
Let $b_0$ be the smallest positive integer such that $b_0\eta$ is in the integer span of $\ba$.
Then there exists a primitive element $b\in \Z^E$ such that
$$b_0\eta + \sum_{e\in E}b_e\ba_e = 0.$$
The vanishing ideal of $\mu_{\ba'}^{-1}(0)\subset S(E')_{u'}$ is generated by the vanishing
ideal of $\mu_{\ba}^{-1}(0)\subset S(E)_u$ along with the additional element $$b_0z_0w_0 + \sum_{e\in E}b_ez_ew_e.$$
Likewise, the subgroup $K'\subset \gm^{E'}$ is generated by the subgroup $K\times\{1\}\subset \gm^E \times\{1\}\subset \gm^{E'}$
along with the image of the element $b'\in\Z^{E'}\cong\Hom(\gm,\gm^{E'})$ obtained by appending $b_0$ to $b$.

Consider a pair of elements $(r,s)\in\Z^E$ such that $r_e<0$ only if $u_e = +$ and $s_e<0$ only if $u_e = -$.
Then $$f_{(r,s)} := \prod_{e\in E}z_e^{r_e}w_e^{s_e}$$ is an element of $k[S(E)_u]$, and such monomials form
a basis for the ring.  The $\gm^E$-weight of $f_{(r,s)}$ is equal to $s-r\in\Z^E\cong\Hom(\gm^E,\gm)$, and so
the weight of $f_{(r,s)}$ under the action of $\gm$ via the map $c:\gm\to\gm^E$ is equal to the dot product $\ell_{(r,s)} := b\cdot(s-r)$.
If $f_{(r,s)}$ is $K$-invariant, then $\ell_{(r,s)}$ is a multiple of $b_0$.
We may now define an $T\times\gm$-equivariant Poisson isomorphism from $k[Y(\ba,u)[\eta]]$ to $k[Y(\ba', u')]$ by sending 
a $K$-invariant monomial $f_{(r,s)}$ to $z_0^{\ell_{(r,s)}/b_0}f_{(r,s)}$.
\end{proof}


\begin{corollary}\label{yau}
If $Z(\ba, u) = Z(\ba', u')$, then
$Y(\ba, u)\cong Y(\ba', u')$ as Poisson $T\times\gm$-varieties.
\end{corollary}

\begin{proof}
If $Z(\ba, u) = Z(\ba', u')$, then the pair $(\ba, u)$ is related to the pair $(\ba', u')$ by a sequence of modifications
consisting of signed permutations and appending
and deleting signed sets of vectors that add to zero.  The fact that this does not change the $T\times\gm$-equivariant
Poisson isomorphism type of the variety then follows from Proposition \ref{ya} and Lemma \ref{append}.
\end{proof}

Given an integral zonotope $Z\subset N_\R$, we define $Y(Z)$ to be the Poisson $T\times\gm$-variety
$Y(\ba,u)$ for any primitive vector configuration $\ba\in N^E$ and sign vector $u\in\{+,-,0\}^E$ such that
$Z(\ba,u) = Z$; Corollary \ref{yau} tells us that this is well defined.  
The following corollaries of Proposition \ref{include} and Lemma \ref{append} are immediate.

\begin{corollary}\label{open}
If $Z\subset N_\R$ is an integral zonotope and $F$ is a face of $Z$, then there is a natural open inclusion
of $Y(F)$ into $Y(Z)$.
\end{corollary}

\begin{corollary}\label{twist}
If $Z\subset N_\R$ is an integral zonotope and $\eta\in N$, then $Y(Z+\eta) \cong Y(Z)[\eta]$.
\end{corollary}

\begin{example}\label{kleinian}
Suppose that $N = \Z$ and $Z = [-r,r]$ for some positive integer $r$.  
Then we may take $E = [r]$ and write $Z = Z(\ba)$, where $\ba = (1,\ldots,1)\in N^E$, and $K\subset \gm^E$
is the kernel of the determinant map to $\gm$.  We then compute
$$k[Y(\ba)] = k[z_1w_1,\ldots,z_rw_r,z_1\cdots z_r, w_1\cdots w_r]/\langle z_iw_i - z_{i+1}w_{i+1}\mid 1\leq i\leq r-1\rangle.$$
Setting $a = z_iw_i$, $b=z_1\cdots z_r$, and $c=w_1\cdots w_r$, we may express this as
$$k[Y(\ba)] = k[a,b,c]/\langle a^r - bc\rangle,$$
where the $T\times\gm$-weights of $a$, $b$, and $c$ are $(0,2)$, $(-1,r)$, and $(1,r)$, respectively.
The Poisson bracket is given by $\{a,b\} = -b$, $\{a,c\} = c$, and $\{b,c\} = ra^{r-1}$.  This is precisely
the Kleinian singularity of type $A_{r-1}$.
\end{example}

The following proposition is straightforward from the definitions.

\begin{proposition}\label{direct sum}
Suppose that $N = N'\oplus N''$ and $Z = Z'\times Z''$ for a pair of integral zonotopes $Z'\subset N_\R'$
and $Z''\subset N_\R''$.  Then $Y(Z)\cong Y(Z')\times Y(Z'')$ as Poisson $T'\times T''\times\gm$-varieties,
where $\gm$ acts diagonally on $Y(Z')\times Y(Z'')$.
\end{proposition}


\begin{example}\label{zero}
Consider the integral zonotope $\{0\}\subset \R$.
We may take $E = [2]$ and write $\{0\} = Z(\ba,u)$, where $\ba = (1,-1)\in \Z^E$ and $u = (+,+)\in\{+,-,0\}^E$.
Then $K\cong\gm$ and $$k[S(E)_u] = k[z_1,z_2,w_1,w_2,z_1^{-1},z_2^{-1}],$$
where $z_i$ has $K$-weight 1 and $w_i$ has $K$-weight -1.  We have $\mu_\ba = z_1w_1 + z_2w_2$,
so $$k[Y(\ba)] = k[z_iw_j, z_1z_2^{-1}, z_1^{-1}z_2]/\langle z_1w_1+z_2w_2\rangle.$$
Setting $a = z_1w_1$ and $b = z_1z_2^{-1}$, we see that this ring is simply isomorphic to $k[a,b,b^{-1}]$,
where the $T\times\gm$ weights of $a$ and $b$ are $(0,2)$ and $(-1,0)$, respectively.
Thus $Y(\{0\})$ is isomorphic to $T^*T\cong T\times\Lie(T)^*$, where $b$ is a coordinate on $T$ and $a$ is a coordinate on
$\Lie(T)^*$.  Here $T$ acts on itself by translation and $\gm$ acts on $\Lie(T)^*$ with weight -2.
More generally, 
$Y(\{0\})\cong T^*T$ for any lattice $N$.
\end{example}

\begin{example}\label{sub}
Let $Z\subset N_\R$ be an integral zonotope centered at the origin.  Let $N'\subset N$
be the intersection of $N$ with the linear span of $Z$, so that $Z$ may also be interpreted
as an integral zonotope in $N'_\R$.
We wish to compare the corresponding varieties $Y_N(Z)$ and $Y_{N'}(Z)$.
Choose a complement $N''$ to $N'$, so that $N \cong N'\oplus N''$.
Then $Z = Z\times\{0\}$, so Proposition \ref{direct sum} and Example \ref{zero} combine to tell us
that $Y_N(Z) \cong Y_{N'}(Z)\times T^*T''$.
\end{example}

\begin{proposition}\label{effective}
For any integral zonotope $Z\subset N_\R$, the action of $T$ on $Y(Z)$ is effective.
\end{proposition}

\begin{proof}
Let $\eta$ be a vertex of $Z$.  By Corollary \ref{open}, $Y(\{v\})$ sits inside of $Y(Z)$ as an open
subvariety.  By Corollary \ref{twist}, $Y(\{v\})$ is $T$-equivariantly isomorphic to $Y(\{0\})$, which is
isomorphic to $T^*T$ by Example \ref{sub}.  Since $T$ acts effectively on $T^*T$, it acts effectively on $Y(Z)$.
\end{proof}

\begin{lemma}\label{cube}
The variety $Y(Z)$ is smooth if and only if $Z$ is a cube.  
\end{lemma}

\begin{proof}
By Corollary \ref{twist}, we may assume that $Z$ is centered at the origin.
By Example \ref{sub}, we may reduce to the case where $\dim Z = \dim N_\R$, and therefore that $Z = Z(\ba)$
and $Y(Z) = Y(\ba)$ for some primitive spanning configuration $\ba\in N^E$.  
If $Z$ is a cube, then $\ba$ is a basis for $N$, and $Y(\ba) = S(E)$.
If $Z$ is not a cube, we will show that origin $0\in Y(\ba)$,
defined as the image of the origin in $S(E)$, is singular.  Indeed, we note that $Y(\ba)$ is a subvariety
of the Lawrence toric variety $X(\ba)$ (Remark \ref{lawrence}) of codimension equal to the dimension of $K$,
and the Zariski tangent space $T_0Y(\ba)$ is a subspace of $T_0X(\ba)$ of the same codimension.  
Since $\ba$ is not a basis for $N$, the theory of toric varieties tells us that $X(\ba)$ is singular at the origin,
thus so is $Y(\ba)$.
\end{proof}

The following lemma will be necessary for the proof of Proposition \ref{reg}.

\begin{lemma}\label{complement}
Suppose that $Z\subset N_\R$ is an integral zonotope with $\dim Z = \dim N_\R$.
Then $$Y(Z)\;\;\smallsetminus \bigcup_{\text{$F\subsetneq Z$ a face}} Y(F) \;\;=\;\; \{0\}.$$
\end{lemma}

\begin{proof}
Write $Z = Z(\ba)$ for some primitive spanning configuration $\ba\in N^E$; then the faces of $Z$
are the zonotopes $Z(\ba,u)$ for sign vectors $u\in\cV^*(\ba)\subset \{+,-,0\}^E$.  
Such a face is proper if and only if $u\neq 0$.
The complement of $Y(\ba,u)$ in $Y(\ba)$ is equal to the image in $Y(\ba)$ of the complement of $S(E)_u$
in $S(E)$ intersected with the zero level of $\mu_\ba$.
We have
$$R(\ba) \;\;:=\;\; \mu_\ba^{-1}(0)\;\;\;\cap \bigcap_{u\in\cV^*(\ba)\smallsetminus\{0\}} \Big(S(E)\smallsetminus S(E)_u\Big)
\;=\; \mu_\ba^{-1}(0)\;\cap\; V(f_u\mid u\in\cV^*(\ba)\smallsetminus\{0\}),$$
and we need to show that $R(\ba)$ is equal to the preimage in $\mu_\ba^{-1}(0)$ of $0\in Y(\ba)$.
In other words, we need to show that $R(\ba)$ is equal to the vanishing locus in $\mu_{\ba}^{-1}(0)$ of the set
of all nonconstant $K$-invariant monomials in $k[S(E)]$.

For any element $r\in\Z^E$, let $$f_r := \prod_{r_e>0}z_e^{r_e}\cdot\prod_{r_e<0}w_e^{r_e}.$$
Then $f_r$ is $K$-invariant if and only if $r\in\Lambda^\perp$, and the ring of $K$-invariant polynomials
is generated by $\{z_ew_e\mid e\in E\}\cup\{f_r\mid r\in \Lambda^\perp\}$.
We recall that 
$$\cV^*(\ba) = \{\sign(m(\ba))\mid m\in N^*\} = \{\sign(r)\mid r\in \Lambda^\perp\},$$ 
so the vanishing locus of the set $\{f_u\mid u\in\cV^*(\ba)\smallsetminus\{0\}\}$
is equal to the vanishing locus of the set $\{f_r\mid r\in \Lambda^\perp\smallsetminus\{0\}\}$.
Hence we only need to show that each $z_ew_e$ vanishes on $R(\ba)$.

Consider the map from $k[x_e\mid e\in E]$ to $k[R(\ba)]$
taking $x_e$ to $z_ew_e$.  The kernel of this map is generated by two families of elements, namely
$$\prod_{u_e\neq 0} x_e\;\;\;\text{for all $u\in\cV^*(\ba)\smallsetminus\{0\}$}
\and
\sum_{e\in E}\lambda_e x_e\;\;\;\text{for all $\lambda\in \Lambda$}.$$
The first family of elements generate the Stanley-Reisner ideal of $\cV(\ba)$, the Gale dual of $\cV^*(\ba)$,
and the second family of elements form a linear system of parameters for the Stanley-Reisner ring.  
This means that the spectrum of the image of this
map has dimension zero, which in turn means that $z_ew_e=0$ vanishes on $R(\ba)$ for all $e\in E$.
\end{proof}

\begin{proposition}\label{reg} The regular locus of $Y(Z)$ is equal to 
$\displaystyle\bigcup_{\substack{\text{$F$ a face of $Z$}\\\text{$F$ a cube}}} Y(F)$. 
\end{proposition}

\begin{proof} If $F$ is a cube, then the open subscheme $Y(F)\subset Y(Z)$ is smooth by Example \ref{sub}
and Lemma \ref{cube}, therefore it is contained in the regular locus of $Y(Z)$.  Conversely, let $y\in Y(Z)$ be given
and let $F$ be the smallest face of $Z$ such that $y\in Y(F)$.
We need to show that, if $F$ is not a cube, then $y$ is a singular point of $Y(F)$.  By Proposition \ref{sub},
we may reduce to the case where $F=Z$ and $\dim Z = \dim N_\R$.  By Lemma \ref{complement}, this implies
that $y=0$.  But we showed in the proof of Lemma \ref{cube} that, when $Z$ is not a cube, the origin is a singular
point of $Y(Z)$.
\end{proof}

We will eventually prove that the Poisson variety $Y(Z)$ is symplectic for any integral zonotope $Z\in N_\R$.
We do not yet have the tools to prove this, but we can prove it for parallelotopes, which we will later use to prove the
general case.

\begin{proposition}\label{paranormal}
If $Z\subset N_\R$ is a parallelotope, then $Y(Z)$ is symplectic.
\end{proposition}

\begin{proof}
Corollary \ref{twist} allows us to reduce to the case where $Z$ is centered at the origin and
Example \ref{sub} allows us to reduce to the case where $\dim Z = \dim N_\R$,
so we may assume that $Z = Z(\ba)$ for some primitive spanning arrangement $\ba$ that is a basis
for $N_\R$.  Then $K$ is finite, and $Y(Z)$ is symplectic by a result of Beauville \cite[2.4]{Beau}.
\end{proof}

\section{Conical symplectic varieties from zonotopal tilings}\label{chp:generalcase}
Once again, let $N$ be a lattice and let $T := \spec k[N^*]$.
Let $Z\subset N_\R$ be an integral zonotope 
and let $\cT$ be a tiling of $Z$.  The purpose of this section will be to define a Poisson 
$T\times\gm$-variety $Y(\cT)$ and show that is a conical symplectic variety with $Y(\cT)_0\cong Y(Z)$.
We will begin by defining a variety $Y(\ba, u, \cM)$ over $Y(\ba,u)$ for any primitive spanning configuration $\ba\in N^E$, sign vector
$u\in\{+,-,0\}^E$, and affine oriented matroid $\cM$ over $\cV^*(\bar\ba)$, where $\bar\ba$ is the restriction of $\ba$
to the set $E_u = \{e\in E\mid u_e = 0\}$.  We will then show that the variety $Y(\ba, u, \cM)$, along with its map to $Y(\ba, u)$,
depends only on the tiling $\cT(\ba, u, \cM)$ of the zonotope $Z(\ba, u)$.  Since the Bohne-Dress theorem \cite[7.32]{Ziegler}
tells us that all tilings arise in this manner, this will allow us to define $Y(\cT)$ for any tiling.

Let $\ba$, $u$, and $\cM$ be given.
For any sign vector $v\in\{+,-,0\}^{E_u}$, let $\tv\in\{+,-,0\}^E$ be the sign vector obtained by appending the nonzero
entries of $u$ to $v$, and let 
$$S(E)_{u,\cM} := \bigcup_{v\in \cMp} S(E)_{\tv}.$$
We define $Y(\ba,u,\cM)$ to be the categorical quotient of $S(E)_{u,\cM}\cap\mu_{\ba}^{-1}(0)$ by $K$.
If $u=0$, we will simply write $S(E)_\cM$ and $Y(\ba,\cM)$.
As in Section \ref{chp:affinecase}, the variety $Y(\ba,u,\cM)$ comes with a Poisson structure, an action of $T\times\gm$,
and a moment map $\mu:Y(\ba,u,\cM)\to\Lie(T)^*$ for the action of $T$.

\begin{remark}\label{Lawrence fan}
As in Remark \ref{lawrence}, we define the {\bf Lawrence toric variety} $X(\ba, u, \cM)$ to be the categorical quotient of $S(E)_{u,\cM}$ by $K$.
If $u=0$, we will simply write $X(\ba,\cM)$.
The moment map $\mu_\ba:S(E)_{u,\cM}\to\Lie(K)^*$ induces a map on $X(\ba,u,\cM)$, and $Y(\ba,u,\cM)$ is isomorphic to the subvariety
of $X(\ba,u,\cM)$ defined by the vanishing of this map.

We now give a combinatorial description of the associated fan.
Consider the lattice $\tilde N$ and the vectors $\rho_e^\pm\in\tilde N$ defined in Remark \ref{lawrence}.
For each $v\in\cMp$, let $\sigma_v\subset \tilde N_\R$ be the cone generated by the vectors 
$\{-\rho_e^\epsilon\mid \epsilon\neq \tv_e\}.$
Then $X(\ba,u,\cM)$ is the toric variety associated with the {\bf Lawrence fan} $\Sigma(\ba,u,\cM)$ consisting of the cones 
$\{\sigma_v\mid v\in\cMp\}$ and all of their faces.  
It is clear that $\Sigma(\ba,u,\cM)$
is a refinement of the Lawrence cone $\sigma(\ba,u)$, all of whose rays are extremal rays of $\sigma(\ba,u)$.  Geometrically,
this means that $X(\ba,u,\cM)$
is proper over the affine Lawrence toric variety $X(\ba,u)$, and that the map from $X(\ba,u,\cM)$ 
induces a bijection between torus orbits of codimension 1 in $X(\ba,u,\cM)$ and $X(\ba,u)$.
Furthermore, every such toric variety arises via this construction.  This is proved when $u = 0$ in \cite[4.16]{Santos}, 
and the general case is similar.
\end{remark}

\begin{remark}\label{GIT}
If $\cM = \cM(\ba,r)$ for an element $r\in \Z^E$, then $S(E)_{\cM}$ is equal to the semistable locus inside of $S(E)$
for the character of $K$ induced by $$r\in\Z^E\cong\Hom(\gm^E,\gm)\twoheadrightarrow\Hom(K,\gm).$$
Thus $X(\ba,\cM)$ is a GIT quotient of $S(E)$ by $K$ 
and $Y(\ba,\cM)$ is a GIT quotient of $\mu_\ba^{-1}(0)\cap S(E)$
by $K$.  This construction coincides with those in \cite[\S 5]{BD} and \cite[\S 6]{HS}.
\end{remark}

\begin{proposition}\label{YbacM}
If $\cM$ and $\cM'$ are affine oriented matroids over $\cV^*(\bar\ba)$ that induce the same tiling $\cT(\ba,u,\cM) = \cT(\ba,u,\cM')$,
then $Y(\ba,u,\cM)\cong Y(\ba,u,\cM')$ as Poisson $T\times\gm$-varieties over $Y(\ba,u)$.
\end{proposition}

\begin{proof}
If $\cT(\ba,u,\cM) = \cT(\ba,u,\cM')$, then $\cM$ and $\cM'$ differ by a signed permutation of $E_u$ that fixes $\bar\ba$.
We will reduce to the case of a single transposition.  Either we have indices $e,f\in E$ with $\ba_e = \ba_f$
and $\cM$ and $\cM$ differ by swapping the indices $e$ and $f$, or we have $\ba_e = -\ba_f$
and $\cM$ and $\cM'$ differ by swapping $e$ and $f$ and reversing the signs.
In the first case, consider the $\gm$-equivariant Poisson automorphism of $S(E)_u$ in which the $e$ and $f$
coordinates are swapped.  This automorphism sends $S(E)_{u,\cM}$ to $S(E)_{u,\cM'}$.
It is not $\gm^E$-equivariant unless we modify the action of $\gm^E$ on 
$S(E)_{u,\cM'}$ by the corresponding automorphism of $\gm^E$.  This automorphism of $\gm^E$ 
takes $K$ to itself, and our isomorphism from $S(E)_{u,\cM}$ to $S(E)_{u,\cM'}$
descends to a $T\times\gm$-equivariant Poisson isomorphism from $Y(\ba,u,\cM)$ to $Y(\ba,u,\cM')$.
The fact that it covers the identity automorphism of $Y(\ba,u)$ follows from the fact that $z_ew_e - z_fw_f$
vanishes on $\mu_\ba^{-1}(0)$.  In the second case, we instead swap $z_e$ with $w_f$ and $z_f$ with $-w_e$
(see Example \ref{tpo}).
\end{proof}

For any tiling $\cT$ of an integral zonotope $Z\subset N_\R$, we choose a finite set $E$,
a primitive spanning configuration $\ba\in N^E$, and a sign vector $u\in\{+,-,0\}^E$ such that $Z=Z(\ba,u)$,
along with an affine oriented matroid $\cM$ over $\cV^*(\bar\ba)$ such that $\cT = \cT(\ba,u,\cM)$,
and we define $Y(\cT)$ to be the Poisson $T\times\gm$-variety $Y(\ba,u,\cM)$ over $Y(\ba,u)=Y(Z)$.
Propositions \ref{yau} and \ref{YbacM} combine to tell us that $Y(\cT)$ is well defined.  

\begin{example}\label{dumb tiling}
If $\cT$ consists only of $Z$ and all of its faces, then $Y(\cT)\cong Y(Z)$.
\end{example}

\begin{example}\label{tpo}
Suppose that $N=\Z$, $Z = [-2,2]$, and $\cT$ is the subdivision of $Z$ into two intervals of length 2.
Let $E = \{1,2\}$ and $\ba = (1,-1)\in N^E$, so that $Z = Z(\ba)$.  Then $K$ 
is equal to the diagonal subtorus of $\gm^2$, which acts on $z_1$ and $z_2$ with weight -1 and on $w_1$ and $w_2$
with weight 1.  We have two choices of oriented matroid that induce this tiling.  The first, which we will call $\cM$,
has maximal positive covectors $\{(+,-),(-,+),(+,+)\}$, so $S(E)_{\cM}$ is the locus of points where $z_1$ and $z_2$
are not both zero.  
The second, which we will call $\cM'$, has maximal covectors $\{(+,-),(-,+),(-,-)\}$, so $S(E)_{\cM'}$ 
is the locus of points where $w_1$ and $w_2$ are not both zero.  
The Lawrence toric variety $X(\ba,\cM)$ is a rank 2 vector bundle over $\mathbb{P}^1$,
and the subvariety $Y(\ba,\cM)$ is a rank 1 sub-bundle which is isomorphic to $T^*\mathbb{P}^1$.
The variety $Y(\ba,\cM')$ is also isomorphic to $T^*\mathbb{P}^1$, and is related to $Y(\ba,\cM)$
by a flop.  The equivariant Poisson isomorphism between $Y(\ba,\cM)$ and $Y(\ba,\cM')$ over $Y(\ba)$ is given by swapping $z_1$ with $w_2$ and $z_2$ with $-w_1$.  Note that the coordinate ring of $Y(\ba)$
is generated by the functions $z_1w_2$, $z_2w_1$, and $z_1w_1 = -z_2w_2$, each of which is preserved by this 
automorphism.  The symplectic form $dz_1\wedge dw_1 + dz_2\wedge dw_2$ is also preserved.
The torus $T$ acts with weight -1 on $z_1$ and $w_2$ and with weight 1 on $z_2$ and $w_1$, so this action
is preserved, as well.
\end{example}

\begin{proposition}\label{limit}
The Poisson $T\times\gm$-variety $Y(\cT)$ is the colimit of the directed system $\{Y(F)\mid F\in \cT\}$
with morphisms given by Corollary \ref{open}.
\end{proposition}

\begin{proof}
It is clear from the definitions that $\{Y(F)\mid F\in \cT\}$ is an open cover of $Y(\cT)$, and therefore
that there is a surjective map from the colimit to $Y(\cT)$ which is locally an isomorphism. 
We need to prove that this map is injective, as well.
We observe that, if $F,F'\in\cT$ and $F\cap F'\neq \emptyset$, then $Y(F)\cap Y(F') = Y(F\cap F')$.
Thus, it is sufficient to show that, for any maximal $F,F'\in\cT$, there exists a sequence
$$F = F_0, F_1, \ldots , F_{m-1}, F_m = F'$$ of maximal elements of $\cT$
such that $F_{i-1}\cap F_i\neq\emptyset$ and $Y(F)\cap Y(F') \subset Y(F_i)$ for all $1\leq 1\leq m$.
Any minimal path of adjacent maximal tiles from $F$ to $F'$ will satisfy this condition.
\end{proof}

The following three results are straightforward from the definition of $Y(\cT)$ and the
results in the previous section (Corollary \ref{twist} and Propositions \ref{effective} and \ref{reg}).

\begin{proposition}\label{tiled twist}
If $\cT'$ is a translation of $\cT$ by $\eta\in N$, then $Y(\cT') \cong Y(\cT)[\eta]$.
\end{proposition}

\begin{proposition}\label{tiling effective}
The $T$-action on $Y(\cT)$ is effective.
\end{proposition}

\begin{proposition}\label{tiling reg}
The regular locus of $Y(\cT)$ is equal to the union of $\{Y(F)\mid \text{$F\in \cT$ a cube}\}.$
\end{proposition}

We are now equipped
to prove that $Y(\cT)$ is symplectic when all of the elements of $\cT$ are parallelotopes.

\begin{proposition}\label{parasymp}
If $F$ is a parallelotope for all $F\in \cT$, then $Y(\cT)$ is symplectic.
\end{proposition}

\begin{proof}
By Proposition \ref{paranormal}, $Y(F)$ is symplectic for all $F\in \cT$.
Then Proposition \ref{limit} tells us that $Y(\cT)$ has an open cover by symplectic varieties.
Since being symplectic is a local property, this implies that $Y(\cT)$ is symplectic.
\end{proof}

We next investigate the relationship between $Y(\cT)$ and $Y(\cT')$ for $\cT'$ a refinement of $\cT$.

\begin{proposition}\label{refinement}
If $\cT$ is a tiling of $Z$ and $\cT'$ is a refinement of $\cT$, then the map from 
$Y(\cT')$ to $Y(Z)$ factors through $Y(\cT)$ via a $T\times\gm$-equivariant Poisson map $Y(\cT')\to Y(\cT)$ which is
proper and an isomorphism over the regular locus.
\end{proposition}

\begin{proof}
We can choose $\ba$, $u$, $\cM$, and $\cM'$ such that $\cT = \cT(\ba,u,\cM)$, $\cT' = \cT(\ba,u,\cM')$,
and $S(E)_{u,\cM'}\subset S(E)_{u,\cM}$; the factorization is induced by this inclusion.
The fact that $Y(\cT')\to Y(\cT)$ is proper follows from the fact that $\Sigma(\ba,u,\cM')$ is a refinement of $\Sigma(\ba,u,\cM)$,
so the map $\Sigma(\ba,u,\cM')\to \Sigma(\ba,u,\cM)$ of Lawrence toric varieties is proper.  The fact that it is an isomorphism
over the regular locus follows from Proposition \ref{tiling reg} and the fact that a cube cannot be integrally subdivided.
\end{proof}

Example \ref{dumb tiling} combines with Proposition \ref{refinement} to show that the affinization of $Y(\cT)$ is isomorphic to $Y(Z)$.

\begin{corollary}\label{affinization}
If $\cT$ is a tiling of $Z$, then $Y(\cT)$ is convex, with $Y(Z)\cong Y(\cT)_0$.
\end{corollary}

\begin{example}\label{armo}
Suppose that $N=\Z$, $Z = [-r,r]$ for some positive integer $r$, and $\cT$ is the tiling of $Z$ by $r$ intervals 
of length 2.  As in Example \ref{kleinian}, we may take $E = [r]$ and
write $Z = Z(\ba)$, where $\ba = (1,\ldots,1)\in N^E$, and $K\subset \gm^E$
is the kernel of the determinant map to $\gm$.  We can then choose $\cM$ with $\cT = \cT(\ba,\cM)$
so that $\cM_+$ is equal to the set of
all sign vectors $v\in \{+,-,0\}^E$ with at most one coordinate of $v$ equal to 0 and with no 
pair of coordinates $i<j$ such that $v_i = -$ and $v_j = +$.  In this case, $S(E)_{\cM}$ is the locus
of points for which we never have $z_i=0=w_j$ for $i<j$.  The variety $Y(\cT)$ is the unique crepant resolution of
the Kleinian singularity $Y(Z)$ of type $A_{r-1}$.  Picking a coarser tiling by deleting a vertex from $\cT$ corresponds to collapsing one of the projective lines in the exceptional fiber of the map $Y(\cT)\to Y(Z)$.
\end{example}

\begin{proposition}\label{symp}
For any integral zonotopal tiling $\cT$, the variety $Y(\cT)$ is symplectic.
\end{proposition}

\begin{proof}
We first prove that $Y(\cT)$ is normal.  Fix an element $F\in\cT$, and
let $\cT'$ be a tiling of $F$ by parallelotopes.  Proposition \ref{parasymp}
tells us that $Y(\cT')$ is symplectic and therefore normal, thus $Y(F) \cong Y(\cT')_0$
is normal, as well.  Since normality is a local property, $Y(\cT)$ is normal.

Now choose a refinement $\cT'$ of $\cT$ such that every $F\in\cT'$ is a parallelotope.
Proposition \ref{parasymp} again tells us that $Y(\cT')$ is symplectic, which means that its
Poisson structure is induced by a symplectic form on the regular locus
which extends to a 2-form on some resolution $\tilde Y$ of $Y(\cT')$.
By Proposition \ref{refinement}, $\tilde Y$ is also a resolution of $Y(\cT)$, and the 
Poisson structure on $Y(\cT)$ is also induced by the restriction of $\tilde\omega$ to the regular locus of $Y(\cT)$.
Thus $Y(\cT)$ is symplectic, as well.
\end{proof}

As a corollary to Proposition \ref{symp}, we may deduce that when every element of $\cT$ is a parallelotope, 
the variety $Y(\cT)$ has $\Q$-factorial terminal singularities.
When $Y(\cT)$ is projective over $Y(Z)$, this has various consequences regarding the birational geometry
of $Y(\cT)$.  For example, it implies that $Y(\cT)$ is ``as smooth as possible" among symplectic partial resolutions of $Y(Z)$.
In particular, if $Y(\cT)$ itself is not smooth (that is, if there exists an element of $\cT$ that is not a cube), then
$Y(Z)$ does not admit any symplectic resolution \cite[25]{Namiflop} (see also \cite[10]{NamikawaNote}).

\begin{corollary}\label{terminal}
If every element of $\cT$ is a parallelotope,
then $Y(\cT)$ has $\Q$-factorial terminal singularities.
\end{corollary}

\begin{proof}
Having $\Q$-factorial terminal singularities is a Zariski local condition, hence it suffices to prove the statement for
$Y(Z)$, where $Z\subset N_\R$ is a parallelotope.  
By Corollary \ref{twist} and Example \ref{sub}, we may reduce
to the case where $Z = Z(\ba)$ for some primitive basis $\ba\in N^E$ for $N_\R$.  As noted in the proof of Proposition
\ref{paranormal}, the group $K$ is finite, and $Y(Z) \cong S(E)/K$ is an orbifold, which has $\Q$-factorial singularities.

By Proposition \ref{symp}, $Y(Z)$ is symplectic, which implies that having terminal singularities is equivalent to 
the singular locus having codimension at least 4
\cite{Nami-note}.  By Proposition \ref{reg}, the regular locus of $Y(Z)$ is equal to 
$$\bigcup_{\substack{\text{$F$ a face of $Z$}\\\text{$F$ a cube}}} Y(F).$$
Since every edge of a parallelotope is a cube, this includes 
$$\bigcup_{\substack{\text{$F$ a face of $Z$}\\ \dim F=1}} Y(F).$$
This in turn is equal to the image in $Y(Z)$ of the union of $S(E)_u$, where $u$ ranges over
all sign vectors $u\in \{+,-,0\}^E$ with at most one zero entry.
The complement of this locus in $S(E)$ is equal to
$$\bigcup_{e\neq e'\in E} V(z_e, w_e, z_{e'}, w_{e'}) \subset S(E),$$
which clearly has codimension 4.
\end{proof}

\section{The extended core}
Let $\cT$ be a tiling of an integral zonotope $Z\subset N_\R$, and let $Y(\cT)$ be the associated
symplectic $T\times\gm$-variety with $\gm$-homogeneous $T$-moment map $\mu:Y(\cT)\to\Lie(T)^*$. 
In this section we will
study the {\bf extended core} $$\Cext(\cT) := \mu^{-1}(0)\subset Y(\cT)$$
and the {\bf core} $$C(\cT)\subset\Cext(\cT),$$ which is defined as the locus of points whose $T$-orbit closures
are proper over a point.
This terminology agrees with the terminology introduced in \cite[\S 2]{HP04} in the case where $\cT = \cT(\ba,r)$ 
for some primitive spanning configuration $\ba\in N^E$ and some element $r\in\Z^E$.
We will use these varieties to prove that $Y(\cT)$ is a hypertoric variety
if and only if $0\in N$ is contained in the interior of $Z$.  We will also explain how to use the extended core to
recover $\cT$ from $Y(\cT)$.

As usual, we write $\cT = \cT(\ba,u,\cM)$ for a primitive spanning configuration $\ba\in N^E$,
a sign vector $u\in\{+,-,0\}^E$, and an affine oriented matroid $\cM$ over $\cV^*(\bar\ba)$.
Then $\Cext(\cT)$ is isomorphic to the categorical quotient of $\mu_E^{-1}(0)\cap S(E)_{u,\cM}$ by $K$.
For each sign vector $v\in\{+,-,0\}^{E_u}$, let $L_v\subset \mu_E^{-1}(0)$ be the linear subspace defined by the
ideal generated by $\{z_e,w_{e'}\mid \tv_e \neq +, \tv_{e'} \neq -\}$, 
and let $C_v$ be the image of $L_v\cap S(E)_{u,\cM}$ in $\Cext(\cT)$.
For every element $F\in\cT$, we define $C_F := C_v$
for the unique $v\in\cMp$ such that $F = Z(\ba,\tv).$


\begin{proposition}\label{cext}
Let $\cT$ be a tiling of a zonotope $Z\subset N_\R$.
\begin{itemize}
\item[i.] The assignment $F\mapsto C_F$ is an inclusion-reversing bijection between $\cT$ and 
closures of $T$-orbits in $\Cext(\cT)$.
In particular, the irreducible components of $\Cext(\cT)$ are $$\{C_F\mid \dim F = 0\}.$$
\item[ii.] For all $F\in\cT$, $C_F$ is isomorphic as a $T$-variety to the toric variety $X(\Sigma_F)$,
where $\Sigma_F$ is the local fan defined in Remark \ref{local fan}.
\item[iii.]  The core $C(\cT)$ is nonempty if and only if $\dim Z = \dim N_\R$, in which case it is 
equal to the union of those $C_F$ for which $F$ is not contained in the boundary of $Z$.
\item[iv.]  For each 0-dimensional zonotope $\{\eta\}\in\cT$, $\gm$ acts on 
the irreducible component $C_{\{\eta\}}$ via the element
$\eta\in N\cong\Hom(\gm,T)$.
\end{itemize}
\end{proposition}

\begin{proof}
We have $C_v\neq\emptyset$ if and only if $L_v\subset S(E)_{u,\cM}$, which is true if and only
if there exists $w\in\cMp$ with $w\leq v$.  In this case, there is a maximal such $w$, and $C_v = C_w$.
This completes the proof of part (i).  

We prove Part (ii) using the description of the Lawrence fan $\Sigma(\ba,u,\cM)$
in Remark \ref{Lawrence fan}.  Recall that $\Sigma(\ba,u,\cM)$ is a fan in $\tilde N_\R$,
and $X(\ba,u,\cM)$ is acted on by $\tilde T = \spec k[\tilde N^*]$.  
If $F = Z(\ba,\tv)$ for some $v\in\cMp$, then $C_F$ is isomorphic to the closed $\tilde T$-subvariety of $X(\ba,u,\cM)$
associated with the cone $\sigma_v\in\Sigma(\ba,u,\cM)$; its fan consists of the cones
$$\sigma_{v,v'} := \R_{\geq 0}(\sigma_{v'}-\sigma_v)\subset \tilde N_\R$$
for all $v'\leq v\in\cMp$.  We have an inclusion of $N$ into $\tilde N$ given by sending $\ba_e$
to $\rho_e^+-\rho_e^-$, and this induces an inclusion of $T$ into $\tilde T$.  The smaller torus
$T$ acts on $C_F$ with a dense orbit, hence $C_F$ is isomorphic
to the $T$-variety associated with the intersection of this fan with $N_\R$.  We have
$$\sigma_{v,v'} \cap N_\R = \R_{\geq 0}\{(v_e' - v_e)\ba_e\mid e\in E\} + \R\{a_e\mid v_e = 0\} = \sigma_{F,F'},$$
where $F' = Z(\ba,\tv')$.  Thus $C_F$ is $T$-equivariantly isomorphic to $X(\Sigma_F)$.

Part (iii) follows from part (ii) and the observation that $\Sigma_F$ is complete if and only if $F$
does not lie in the boundary of $Z$.
For part (iv), we may write $\{\eta\} = F(\ba,\tv)$ for some maximal $v\in\cMp$.
The action of $\gm$ on $L_{\tv}$ is induced by the cocharacter $\tv\in\Z^E\cong\Hom(\gm,\gm^E)$, 
which descends to the cocharacter
$\eta = \sum \tv_e\ba_e \in N\cong\Hom(\gm,T)$.
\end{proof}

\begin{example}
Let $\cT$ be the tiling in Example \ref{armo}, consisting of $r$ adjacent intervals of length 2.
By Proposition \ref{cext}, each of the intervals corresponds to a $T$-fixed point,
each of the internal vertices corresponds to a projective line connecting two adjacent fixed points,
and the two external vertices correspond to affine lines.  Thus $C(\cT)$ is isomorphic to a chain
of $r-1$ projective lines, and $\Cext(\cT)$ is obtained from $C(\cT)$ by adding an affine line at either end of the chain.
\end{example}

\begin{example}
Let $\cT$ be either of the two tilings of a hexagon by cubes in Example \ref{ptwo}.  
The variety $Y(\cT)$ is isomorphic to $T^*\mathbb{P}^2$.  The core $C(\cT)$ is equal to the zero section, which is represented
by the single internal vertex.  The extended core consists of the conormal varieties to the toric strata of $\mathbb{P}^2$.
The three vertices which lie on only two edges represent the conormal varieties to the three fixed points, each of which is isomorphic to $\bA^2$.
Three three vertices which lie on three edges represent the conormal varieties to the three coordinate lines, each of which is isomorphic
to the blow-up of $\bA^2$ at the origin.
\end{example}

\begin{corollary}\label{reconstruct}
If $\cT$ and $\cT'$ are tilings in $N_\R$, then $Y(\cT)$ is $T\times\gm$-equivariantly isomorphic to $Y(\cT')$
if and only if $\cT = \cT'$.
\end{corollary}

\begin{proof}
By parts (i) and (iv) of Proposition \ref{cext}, we can recover the vertices of $\cT$ by looking at the action of $\gm$ on the irreducible components
of the extended core.  By part (ii), we can recover the fans at these vertices in terms of the action of $T$, and this completely determines the tiling $\cT$.\end{proof}

\begin{corollary}\label{pos wts}
Let $\cT$ be a tiling of an integral zonotope $Z\subset N_\R$.
The $\gm$-variety $Y(\cT)$ has non-negative weights if and only if $0\in Z$,
and it has positive weights if and only if $0$ is contained in the interior of $Z$.
\end{corollary}

\begin{proof}
Since the moment map $\mu$ has weight 2,
$Y(\cT)$ has non-negative (respectively positive) weights if and only if $\Cext(\cT) = \mu^{-1}(0)$ has non-negative (respectively positive) weights.  
This in turn is the
case if and only if each irreducible component of $\Cext(\cT)$ has non-negative (respectively positive) weights.  
By Proposition \ref{cext},
the irreducible components of $\Cext(\cT)$ are indexed by the vertices of $\cT$.
Given a vertex $\eta$, the component $C_{\{\eta\}}$ is $T$-equivariantly isomorphic to the toric variety $X(\Sigma_{\{\eta\}})$,
and $\gm$ acts via $-\eta\in N\cong\Hom(\gm,T)$.  This action has non-negative weights if and only if $-\eta$ lies in
the support of $\Sigma_{\{\eta\}}$, which is equivalent to the statement that 0 lies in $\eta + |\Sigma_{\{\eta\}}|$.
This holds for every vertex of $\cT$ if and only if 0 lies in $Z$.
The action has positive weights if and only if $-\eta$ lies in the interior of
the support of $\Sigma_{\{\eta\}}$, which is equivalent to the statement that 0 lies in the interior of $\eta + |\Sigma_{\{\eta\}}|$.
This holds for every vertex of $\cT$ if and only if 0 lies in the interior of $Z$.
\end{proof}

\begin{example}
Suppose that $N=\Z$ and $Z = [\ell-r,\ell+r]$ for a positive integer $r$ and an arbitrary integer $\ell$.
When $\ell=0$, we showed in Example \ref{kleinian} that $k[Y(Z)] \cong k[a,b,c]/\langle a^r-bc\rangle$, 
where the $T\times\gm$-weights of $a$, $b$, and $c$ are $(0,2)$, $(-1,r)$, and $(1,r)$, respectively.
More generally, Corollary \ref{twist} tells us that the $T\times\gm$-weights of $a$, $b$, and $c$
are $(0,2)$, $(1,r+\ell)$, and $(-1,r-\ell)$.  The $\gm$-weights are all positive if and only if $|\ell|<r$,
which is equivalent to the condition that 0 lies in the interior of $Z$.
\end{example}

Corollary \ref{pos wts} provides the last ingredient in the proof of the main result of this paper.

\begin{theorem}\label{main}
Let $Z\in N_\R$ be an integral zonotope with $0$ in its interior and $\cT$ a tiling of $Z$.
Then $Y(\cT)$ is a $T$-hypertoric variety.
\end{theorem}

\begin{proof}
By construction, $Y(\cT)$ is a Poisson $T\times\gm$-variety of dimension $2\dim T$, the action of $T$ is Hamiltonian,
and the Poisson bracket has weight -2 with respect to the action of $\gm$.  The action of $T$ is effective by Proposition \ref{tiling effective}.
Corollary \ref{affinization} and Proposition \ref{symp} tell us that $Y(\cT)$ is conical symplectic variety.
Finally, Corollary \ref{pos wts} tells us that the action of $\gm$ has positive weights, which completes the proof.
\end{proof}

\begin{conjecture}\label{all}
Let $Y$ be a $T$-hypertoric variety.  There exists an integral zonotopal tiling $\cT$ in $N_\R$ such that
$Y$ is isomorphic to $Y(\cT)$ as a Poisson $T\times\gm$-variety.
\end{conjecture}

\begin{remark}
Though we do not have a proof of Conjecture \ref{all}, we briefly outline two possible approaches,
based on two different (conjectural) ways to obtain a zonotopal tiling from an abstract $T$-hypertoric variety.
\\\\
{\bf First approach.}
For any $T\times\gm$-variety $U$, let 
$$Z(U) := \operatorname{Conv}\{\eta\in N\mid \text{$U[\eta]$ has non-negative weights}\}.$$
We define a {\bf \boldmath{$T$}-quasihypertoric variety} to be a symplectic, convex $T\times\gm$-variety $U$
such that that action of $T$ is effective and Hamiltonian and $Z(U)\neq\emptyset$.
The proof of Theorem \ref{main}, along with Corollary \ref{twist}, demonstrates that
$Y(\cT)$ is $T$-quasihypertoric for any zonotopal tiling $\cT$.  
Furthermore, we conjecture that every affine $T$-quasihypertoric variety is of this form $Y(Z)$ for a single zonotope $Z$, 
and that all morphisms between affine $T$-quasihypertoric varieties are of the form $Y(\ba,u)\to Y(\ba,u')$
for some $u'\leq u$, induced by the inclusion of $S(E)_u$ into $S(E)_{u'}$.

Let $Y$ be a $T$-hypertoric variety, and let $Z := Z(Y)$.
Consider the set $\mathcal{S}$ of all affine $T$-quasihypertoric subvarieties of $Y$, and let 
$\cT := \{Z(U)\mid U\in \mathcal{S}\}$.  Then $\cT$ is a collection of zonotopes, each contained in $Z$,
any pair of which intersects in a face or not at all.  We conjecture that $\cT$ is in fact a tiling of $Z$
(that is, that every element of $Z$ is contained in some element of $\cT$), and that $Y\cong Y(\cT)$.
\\\\
{\bf Second approach.}
Given a $T$-hypertoric variety $Y$, let $\mu:Y\to\Lie(T)^*$ be the unique $\gm$-homogeneous $T$-moment map. 
We may define the extended core $\Cext := \mu^{-1}(0)\subset Y$ as before, 
and it is necessarily a finite union Lagrangian $T\times\gm$-subvarieties.
For each irreducible component $C\subset \Cext$, $\gm$ must act via a cocharacter $\eta_C\in\Hom(\gm,T)\cong N$.
We conjecture that there exists an integral zonotopal tiling $\cT$ in $N_\R$ with vertex set equal
to $\{\eta_C\mid \text{$C\subset \Cext$ a component}\}$ such that $C$ is $T$-equivariantly isomorphic to
$X(\Sigma_{\{\eta_C\}})$, and that $Y\cong Y(\cT)$.
\end{remark}

\section{Divisors and line bundles}
Let $Z\subset N_\R$ be an integral zonotope with 0 in its interior, and let $\cT$ be a tiling of $Z$.
The purpose of this section is to study $T$-equivariant divisors and line bundles on $Y := Y(\cT)$ in combinatorial terms.
Since we will not be concerned with the action of $\gm$ in this section, we will assume for simplicity that $Z$ is centered
at the origin, so that we may write $Z = Z(\ba)$ for some primitive spanning arrangement $\ba\in N^E$.
We may also write $\cT = \cT(\ba,\cM)$ for some affine oriented matroid $\cM$ over $\cV^*(\ba)$.

We begin by studying the Weil divisor class group $\Cl(Y)$,
along with the $T$-equivariant analogue $\Cl_T(Y) := A^T_{\dim Y - 1}(Y)$.  We note that the forgetful map $\Psi:\Cl_T(Y)\to\Cl(Y)$
is surjective, with kernel isomorphic to $N^*$ \cite[2.3]{Brion-Chow}.
For each element $e\in E$, let $D^+_e\subset Y$ be the divisor defined by the vanishing of $z_e$;
that is, $D^+_e$ is the image of $V(z_e)\cap S(E)_\cM \cap \mu_\ba^{-1}(0)$ in $Y$.
Similarly, let $D_e^-\subset Y$ be the divisor defined by the vanishing of $w_e$.
Since the product $z_ew_e$ is a $T$-invariant function on $Y$, we have $[D_e^+] + [D_e^-] = 0 \in \Cl_T(Y)$.
Consider the homomorphism $$\Phi:\Z^E\to \Cl_T(Y)$$ taking the
standard generators of $\Z^E$ to $\{[D_e^+]\mid e\in E\}$.

\begin{proposition}\label{generators}
The map $\Phi$ is an isomorphism, and it takes $N^*\cong\Lambda^\perp\subset\Z^E$
to $\ker(\Psi)$.
\end{proposition}

\begin{proof}
We first prove surjectivity.
Since $[D_e^+] + [D_e^-] = 0$, this is equivalent to the statement that $\Cl_T(Y(\cT))$ 
is generated by $\{[D^\pm_e]\mid e\in E\}$, which is in turn equivalent to the statement that
$\Cl_T\!\big(Y\smallsetminus \bigcup_{e\in E} D^\pm_e\big)$ is trivial.
The group $T$ acts freely on $Y\smallsetminus \bigcup_{e\in E} D^\pm_e$, 
so we need only show that the Weil divisor class group of the quotient is trivial.
The is isomorphic to the subvariety of the torus with coordinates $\{z_ew_e\mid e\in E\}$ defined by the
vanishing of $\mu_\ba$.  In particular, it is isomorphic to an open subspace of an affine
space, and therefore has trivial Weil divisor class group.  Thus $\Phi$ is surjective.

To prove injectivity, let $F\in\cT$ be a vertex and let $C_F\subset Y$ be the corresponding
extended core component.  Then $\Cl_T(C_F)$ has a basis indexed by the rays of $\Sigma_F$ \cite[4.1.3]{CLS},
which are in bijection by Proposition \ref{cext}(ii) with the edges of $\cT$ incident to $F$.
If $F' = Z(\ba,u')$ is an edge incident to $F$, then 
the map from $\Cl_T(Y)$ to $\Cl_T(C_F)$ takes $[D_e^+]$ to the basis element indexed by $F'$
if $u'_e = 0$ and to zero otherwise.
For every $e\in E$, there exists
a positive covector $u'\in\cMp$ with $u'_e\neq 0$ and $Z(\ba,u')$ an edge incident to $F$, so
the composition
$$\Z^E\to \Cl_T(Y)\to\bigoplus_{\substack{F\in\cT\\ \dim F = 0}}\Cl_T(C_F)$$
is injective.  Since the first map in this composition is equal to $\Phi$, $\Phi$ is injective, as well.

An element of $\Z^E$ maps to zero in $\Cl(Y)$ if and only if it maps to zero in $\Cl(C_F)$ for every $F$.
Again by \cite[4.1.3]{CLS}, the set of elements with this property is equal to $\Lambda^\perp$.
\end{proof}

We next study the equivariant Picard group $\Pic_T(Y)$, which is isomorphic to the class group of $T$-equivariant
Cartier divisors.  By normality of $Y$, $\Pic_T(Y)$ is a subgroup of $\Cl_T(Y)$.  For any $r\in\Z^E$, consider the element
$\Phi(r) = \sum r_e [D_e^+]\in \Cl_T(Y)$.  We first consider the case where $\cT$ consists only of $Z$ and its faces,
so that $Y = Y(Z)$.

\begin{lemma}\label{base}
We have $\Pic_T(Y(Z))\cong N^*\cong\Lambda^\perp\subset \Z^E\cong \Cl_T(Y(Z))$.
\end{lemma}

\begin{proof}
We follow the proof of the analogous statement for affine toric varieties \cite[4.2.2]{CLS},
making appropriate modifications.  By Proposition \ref{generators} and the fact that $[D_e^+] + [D_e^-] = 0$,
every $T$-equivariant Weil divisor is equivalent to an effective divisor of the form
$$D = \sum_{e\in E}r_e^+ D_e^+ + \sum_{e\in E}r_e^- D_e^-$$
for some $r_e^+,r_e^-\in\mathbb{N}^E$.
Assume that $D$ is Cartier.  This means that $D$ is locally principal; in particular, there exists an open neighborhood
$U$ of $0\in Y(Z)$ and a rational function $f$ on $U$ such that $D|_U$ is equal to the divisor associated with $f$.  Since $D$ is effective, we may assume
that $f$ extends to a regular function on $Y(Z)$, which we will also denote by $f$, and we now have $\operatorname{div}(f)\geq D$
with equality after restricting to $U$.  This implies that $f$ is contained in the ideal
$$I(D) := \Big\langle z^{s^+}w^{s^-}\;\Big{|}\; \text{$s^+\geq r^+$, $s^-\geq r^-$, and $s^+-s^-\in N^*$}\Big\rangle\subset k[Y(Z)].$$
Since $\operatorname{div}(f)|_U = D|_U$, the rational function $g/f$ is regular on $U$ for any $g\in I(D)$.
Since $f\in I(D)$, there exists a generator $z^{s^+}w^{s^-}\in I(D)$ such that $z^{s^+}w^{s^-}/f$ does is nonvanishing in some
neighborhood $V$ of $0\in U$, so that $D|_V = \div(f)|_V = \div(z^{s^+}w^{s^-})|_V$.
But $$\div(z^{s^+}w^{s^-}) = \sum_{e\in E}c_e^+ D_e^+ + \sum_{e\in E}c_e^- D_e^-,$$
and if this divisor coincides with $D$ in a neighborhood of 0, then it coincides with $D$ globally.  Hence $D$ is principal,
and we have $[D]\in\ker(\Psi) = \Lambda^\perp$.  Conversely, any element of $\Lambda^\perp \cong N^*$ 
determines a linearization of the trivial bundle on $Y(Z)$,
and therefore lies in $\Pic_T(Y(Z))$.
\end{proof}

Recall from Section \ref{sec:combinatorics} sublattice $P_\cT\subset\Z^E$, defined in Equation \eqref{pct}, which is isomorphic 
via the map $r\mapsto\varphi(\ba,r)$ to the group $\SF(\cT)$ of support functions on $\cT$.

\begin{proposition}\label{support}
For any $r\in\Z^E$, the equivariant Weil divisor $\Phi(r)$ is Cartier if and only if $r\in P_\cT$.
Thus we have an isomorphism $$\Pic_T(Y)\cong P_\cT\cong \SF(\cT).$$
\end{proposition}

\begin{proof}
The property of being Cartier is local, so it is sufficient to consider the restriction of $\Phi(r)$ to $Y(\ba,u)\subset Y(\cT) = Y$
for each positive covector $u\in\cMp$.  We have
$$Y(\ba,u) \cong Y \smallsetminus \left(\bigcup_{u_e=+}D_e^+\cup\bigcup_{u_e=-}D_e^-\right),$$
hence we have a commutative diagram
\[\tikz[->,thick]{
\matrix[row sep=10mm,column sep=20mm,ampersand replacement=\&]{
\node (a) {$\Z^E$}; \& \node (b) {$\Z^{E_u}$}; \\
\node (c) {$\Cl_T(Y)$}; \& \node (d) {$\Cl_T(Y(\ba,u))$}; \\
};
\draw (a) -- (b) ; 
\draw (a) -- (c) node[ left,midway]{$\Phi_\cT$};
\draw (c) -- (d) ; 
\draw (b) -- (d) node[ right,midway]{$\Phi_Z$}; 
}\]
in which the two arrows is the standard projection.
By Lemma \ref{base}, the restriction of $\Phi(r)$ to $Y(\ba,u)$ is Cartier if and only if the projection of $r$ onto $\Z^{E_u}$
lies in $N^*\cong \Lambda_u^\perp$.  This holds for every positive covector $u\in\cMp$ if and only if $r\in P_\cT$.
\end{proof}


Let $X = X(\ba, \cM)$ be the Lawrence toric variety defined in Remark \ref{Lawrence fan}.
Recall that $X$ is acted on by the torus $\tilde T = \spec k[\tilde N^*]$, and the inclusion of $N$
into $\tilde N$ that sends $\ba_e$ to $\rho_e^+-\rho_e^-$ induces an inclusion of $T$ into $\tilde T$.
We have a commutative diagram
\[\tikz[->,thick]{
\matrix[row sep=10mm,column sep=20mm,ampersand replacement=\&]{
\node (a) {$\Z^E\oplus\Z^E$}; \& \node (b) {$\Cl_{\tilde T}(X)$}; \\
\node (c) {$\Z^E$}; \& \node (d) {$\Cl_T(Y)$}; \\
};
\draw (a) -- (b)node[ above,midway]{$\tilde\Phi$} ; 
\draw (a) -- (c) ;
\draw (c) -- (d) node[ above,midway]{$\Phi$}; 
\draw (b) -- (d) ; 
}\]
in which the horizontal arrows are isomorphisms, the left vertical arrow sends
$r^\pm$ to $r^+-r^-$, and the right vertical arrow is given by taking a $\tilde T$-invariant divisor
on $X$, regarding it as a $T$-invariant divisor, and intersecting with $Y$.  
Lemma \ref{same pic} tells us
that an element of $\Cl_{\tilde T}(X)$ is Cartier if and only if its image in $\Cl_T(Y)$ is Cartier;
equivalently, it says that the restriction map $\Pic(X)\to\Pic(Y)$ on non-equivariant Picard groups
is an isomorphism.

\begin{lemma}\label{same pic}
For any $r^\pm\in\Z^E\oplus\Z^E$, $\tilde\Phi(r^\pm)$ is Cartier if and only if $\Phi(r^+-r^-)$ is Cartier.
\end{lemma}

\begin{proof}
We know from Proposition \ref{support} that $\Phi(r^+-r^-)$ is Cartier if and only if $r^+-r^-\in P_\cT$;
that is, if and only if it can be used to define a support function on $\cT$.  Similarly, $\tilde\Phi(r^\pm)$
is Cartier if and only if it can be used to define a support function on the Lawrence fan $\Sigma$ \cite[4.2.12]{CLS}.
More precisely, for each positive covector $u\in\cMp$, we have the cone $\sigma_u\in\Sigma$
generated by $\{-\rho_e^\epsilon\mid \epsilon\neq u_e\}$ (Remark \ref{Lawrence fan}), and we would like to define
a linear function $\tilde\varphi(\ba,r^\pm)$ on $\sigma_u$ by sending $\rho_e^+$ to $b_e^+$ if $u_e\neq +$
and $\rho_e^-$ to $b_e^-$ if $u_e\neq -$.  This is well-defined if and only if $r^+-r^-\in P_\cT$.
\end{proof}

We next use Lemma \ref{same pic} to identify the nef and ample cones in $\Pic_T(Y)$.
For any $r^\pm\in\Z^E\oplus\Z^E$ such that $r^+-r^-\in P_\cT$, 
let $\tilde\varphi(\ba,r^\pm)$ be the support function on $\Sigma$ defined
in the proof of Lemma \ref{same pic}.  Similarly, for any $r\in P_\cT\subset \Z^E$, let $\varphi(\ba,r)$ be the support function
on $\cT$ defined in Section \ref{sec:combinatorics}.

\begin{proposition}\label{ample}
For any $r^\pm\in\Z^E\oplus\Z^E$ such that $r^+-r^-\in P_\cT$, the following are equivalent:
\begin{itemize}
\item[i.] The support function $\tilde \varphi(\ba,r^\pm)$ on $\Sigma$ is convex (respectively strictly convex)
\item[ii.] The line bundle $\tilde\Phi(r^\pm)$ on $X$ is nef (respectively ample)
\item[iii.] The support function $\varphi(\ba,r^+-r^-)$ on $\cT$ is convex (respectively strictly convex)
\item[iv.] The line bundle $\Phi(r^+-r^-)$ on $Y$ is nef (respectively ample).
\end{itemize}
\end{proposition}

\begin{corollary}\label{bundles}
Under the isomorphism $\Pic_T(Y)\cong P_\cT\cong \SF(\cT)$ of Proposition \ref{support},
\begin{itemize}
\item (non-equivariantly) trivial line bundles correspond to linear functions
\item nef line bundles correspond to convex functions, and 
\item ample line bundles correspond to strictly convex functions.
\end{itemize}
\end{corollary}

\begin{proofample}
The equivalence of (i) and (ii) is a standard fact about toric varieties; see \cite[6.1.7 \& 6.3.12]{CLS}
for numerical effectiveness and \cite[7.2.6]{CLS} for ampleness.  The fact that (ii) implies (iv) is automatic,
since the properties of being nef or ample are preserved when restricting to a closed subvariety.
To see that (iv) implies (ii), we note that a line bundle on a toric variety is nef (respectively ample)
if and only if the degree of its restriction to any complete, irreducible, torus-invariant curve is non-negative (respectively positive)
\cite[6.3.12 \& 15.5.1]{CLS}.  But all such curves are contained in $Y$; they are exactly the curves $C_F$ where $F\in\cT$
has codimension 1 and is not contained in the boundary of $Z$.  If $\Phi(r^+-r^-)$ is nef (respectively ample),
then its restriction to any of these curves has non-negative (respectively positive) degree, which implies that 
$\tilde\Phi(r^\pm)$ is nef (respectively ample).  

It remains only to prove the equivalence of (i) and (iii).  We first give an explicit
characterization of convexity of the function $\varphi(\ba,r)$ for any $r\in P_\cT$.  Let $v,v'\in\cMp$ be minimal positive covectors
such that $Z(\ba,v)$ and $Z(\ba,v')$ intersect in a face of codimension 1, and let $\eta$ be any point
in $Z(\ba,v)$ that does not lie in $Z(\ba,v')$.  Then we can write
$$\sum_{e\in E} t_e\ba_e \;\;=\;\; \eta \;\;=\;\; \sum_{e\in E}t_e'\ba_e$$
with $$t_e\in
\begin{cases}
[-1,1] \;\;\;\text{if $v_e = 0$}\\
\{1\} \;\;\;\;\;\;\;\,\text{if $v_e = +$}\\
\{-1\} \;\;\;\;\;\text{if $v_e = -$}
\end{cases}
\and\;\;
t_e' \in 
\begin{cases}
\R \;\;\;\;\;\;\;\;\;\;\text{if $v'_e = 0$}\\
\{1\} \;\;\;\;\;\;\;\,\text{if $v'_e = +$}\\
\{-1\} \;\;\;\;\;\text{if $v'_e = -$}
\end{cases}.$$
The choices of $t$ and $t'$ are not unique (unless $Z(\ba,v)$ and $Z(\ba,v')$ are parallelotopes),
but for any $r\in P_\cT$, the quantities $r\cdot t$ and $r\cdot t'$ do not depend on the choices.
(For example, $r\cdot t$ is equal to $\varphi(\ba,r)(\eta)$.)
The restriction of $\varphi(\ba,r)$ to $Z(\ba,v)\cup Z(\ba,v')$ is convex (respectively strictly convex)
if and only if $r\cdot (t'-t)$ is non-negative (respectively positive).  The function $\varphi(\ba,r)$ is convex
on all of $Z$ if and only if this condition is satisfied for all such pairs $v$ and $v'$.

We now give a similar criterion for convexity of $\tilde\varphi(\ba,r^\pm)$ for any $r^\pm\in\Z^E\oplus\Z^E$
such that $r^+-r^-\in P_\cT$.  
Let $v,v'\in\cMp$ be minimal positive covectors
such that $\sigma_v$ and $\sigma_{v'}$ intersect in a face of codimension 1; this is equivalent to the condition
that $Z(\ba,v)$ and $Z(\ba,v')$ intersect in a face of codimension 1.  Let $\delta$ be any point in $\sigma_v$
that does not lie in $\sigma_{v'}$.
Then we can write
\begin{equation}\label{delta}\sum_{e\in E} s^+_e\rho_e^+ + \sum_{e\in E} s^-_e\rho_e^- \;\;=\;\; \delta \;\;=\;\; 
\sum_{e\in E} (s^+_e)'\rho_e^+ + \sum_{e\in E} (s^-_e)'\rho_e^-\end{equation}
with \begin{equation}\label{s}s_e^\epsilon\in
\begin{cases}
\R_{\leq 0} \;\;\;\;\;\;\,\text{if $v_e \neq \epsilon$}\\
\{0\} \;\;\;\;\;\;\;\,\text{if $v_e = \epsilon$}
\end{cases}
\and\;\;
(s_e^\epsilon)' \in 
\begin{cases}
\R \;\;\;\;\;\;\;\;\;\;\text{if $v'_e \neq \epsilon$}\\
\{0\} \;\;\;\;\;\;\;\,\text{if $v'_e = \epsilon$}
\end{cases}.\end{equation}
Once again, the choices of $s^\pm$ and $(s^\pm)'$ are not unique, but if $r^+-r^-\in P_\cT$,
then $r^\pm\cdot s^\pm$ and $r^\pm\cdot (s^\pm)'$ do not depend on these choices.
The restriction of $\tilde\varphi(\ba,r^\pm)$ to $\sigma_v\cup \sigma_{v'}$ is convex (respectively strictly convex)
if and only if $r^\pm\cdot ((s^\pm)'-s^\pm)$ is non-negative (respectively positive), and the function $\tilde\varphi(\ba,r^\pm)$ is convex on all of $|\Sigma|$ if and only if this condition is satisfied for all such pairs $v$ and $v'$.

Suppose that $v$ and $v'$ are given and that we have already chosen $\eta$, $t$, and $t'$ as above.
We will now explain how to choose $\delta$, $s^\pm$, and $(s^\pm)'$.  We start by choosing $s^\pm$ and $(s^\pm)'$ as follows:
\begin{itemize}
\item If $v_e' = 0$, put $s_e^+ = s_e^- = 0$ and $(s_e^+)' = -(s_e^-)' = t_e' - t_e\in\R$.
\item If $v_e' = +$, put $(s_e^+)' = s_e^- = 0$ and $s_e^+ = (s_e^-)' = t_e - t_e' = t_e - 1\in [-2,0]\subset\R_{\leq 0}$.
\item If $v_e' = -$, put $(s_e^-)' = s_e^+ = 0$ and $s_e^- = (s_e^+)' = t_e' - t_e = -1 - t_e\in [-2,0]\subset\R_{\leq 0}$.
\end{itemize}
We then observe that the conditions of Equation \eqref{s} are satisfied, and that 
$$\sum_{e\in E} s^+_e\rho_e^+ + \sum_{e\in E} s^-_e\rho_e^- \;\;=\;\;
\sum_{e\in E} (s^+_e)'\rho_e^+ + \sum_{e\in E} (s^-_e)'\rho_e^-,$$
so we may use Equation \ref{delta} to define $\delta$.  It is clear from the first equality in Equation \ref{delta}
that $\delta\in\sigma_v$.  One can also check that the condition
$\eta\notin Z(\ba,v')$ implies that $\delta\notin\sigma_{v'}$.
Thus
$$
\text{$\tilde\varphi(\ba,r^\pm)$ is convex on $\sigma_v\cup\sigma_{v'}$\;\;\;\; $\Leftrightarrow$\;\;\;\;
$r^\pm\cdot ((s^\pm)'-s^\pm)$ is non-negative}
$$
and
$$
\text{$\varphi(\ba,r^+-r^-)$ is convex on $Z(\ba,v)\cup Z(\ba,v')$\;\;\;\; $\Leftrightarrow$\;\;\;\;
$(r^+-r^-)\cdot (t'-t)$ is non-negative}.
$$
But $r^\pm\cdot ((s^\pm)'-s^\pm) = (r^+-r^-)\cdot (t'-t)$, so the two conditions are identical.
The same statements hold with strict convexity and positivity, thus the equivalence of (i) and (iii) is proved.
\end{proofample}
 
\begin{example}
Consider the hypertoric variety $Y(\ba,\cM)\cong T^*\mathbb{P}^1$ from Example \ref{tpo}.
A support function $\varphi$ on $\cT$ is convex if and only if $\varphi(0)\geq 0$, and it is strictly
convex if and only if $\varphi(0)>0$.  Since $\{0\} = Z(\ba,(+,+))$, we have
$\varphi(\ba,r)(0) = r_1 + r_2$ for any $r = (r_1,r_2)\in \Z^2$.
Thus $\Phi(r)$ is nef if and only if $r_1+r_2\geq 0$, and ample if and only if $r_1+r_2>0$.
Furthermore, we have $r_1+r_2=0$ if and only if $r\in\Lambda^\perp$ if and only if $\Phi(r)$ is
a linearization of the trivial bundle.
\end{example}

As a corollary to Proposition \ref{ample}, we may deduce that $Y(\cT)$ is projective over $Y(Z)$ if and only if $\cT$
admits a strictly convex support function; that is, if and only if $\cT$ is regular.

\begin{corollary}\label{regproj}
The variety $Y(\cT)$ is projective over $Y(Z)$ if and only if $\cT$ is regular.
\end{corollary}

\begin{example}\label{bad component}
Given any complete rational cone $\Sigma$ in $N_\R$, one can find a tiling $\cT$
and a vertex $F\in\cT$ such that $\Sigma_F = \Sigma$, which implies by Proposition \ref{cext}(ii)
that the toric variety $X(\Sigma)$ is a component of the core of $Y(\cT)$.  As we explained in Remark \ref{local fan},
if $\Sigma$ does not admit
a strictly convex support function (that is, if $X(\Sigma)$ is not projective over $X(\Sigma)_0$), then neither does $\cT$
(that is, $Y(\cT)$ is not projective over $Y(\cT)_0=Y(Z)$).
\end{example}

\begin{example}
Example \ref{bad component} gives one obstruction to $Y(\cT)$ being projective over $Y(Z)$, namely the existence
of a nonprojective core component.  However, it is possible for $Y(\cT)$ is not projective over $Y(Z)$ even if all
of the core components are projective.  Indeed,
let $\cT$ be the tiling in Figure \ref{irregular}.  The tiling $\cT$ is cubical, so $Y(\cT)$ is smooth
by Proposition \ref{tiling reg}.  All of the core components of $Y(\cT)$ are toric surfaces,
and therefore projective.  But $\cT$ is not regular, so $Y(\cT)$ is not projective over $Y(Z)$.
\end{example}

\bibliographystyle{alpha}
\bibliography{symplectic}

\end{document}